\documentclass[10pt]{amsart}
\usepackage{amsfonts}
\usepackage{amscd}
\usepackage{amssymb}
\usepackage{amsthm}
\usepackage{amsmath}
\usepackage{graphicx}
\usepackage{color}
\usepackage{hyperref}
\usepackage[all]{xy}
\xyoption{matrix}
\xyoption{arrow}

\def\mesh{\rm mesh}

\def\ot{\leftarrow}

\def\xrarrow{\xrightarrow} 
\def\xlarrow{\xleftarrow} 

\def\<{\left<}
\def\>{\right>}

\DeclareMathOperator{\End}{End}
\DeclareMathOperator{\Hom}{Hom}%

\newcommand{\field}[1]{\mathbb{#1}}
\newcommand{\ZZ}{\ensuremath{{\field{Z}}}}
\newcommand{\CC}{\ensuremath{{\field{C}}}}
\newcommand{\RR}{\ensuremath{{\field{R}}}}

\def\ll{\lambda}
\def\LL{\Lambda}

\newcommand{\cB}{\ensuremath{{\mathcal{B}}}}
\newcommand{\cC}{\ensuremath{{\mathcal{C}}}}
\newcommand{\cD}{\ensuremath{{\mathcal{D}}}}

\newcommand{\cF}{\ensuremath{{\mathcal{F}}}}

\newcommand{\cX}{\ensuremath{{\mathcal{X}}}}

\def\a{\alpha}
\def\b{\beta}
\def\g{\gamma}
\def\d{\partial}

\def\e{\epsilon}
\def\f{\phi}
\def\k{\kappa}

\def\s{\sigma}

\def\t{\tau}
\def\th{\theta}

\def\w{\omega}

\def\ot{\leftarrow}

\def\onto{\twoheadrightarrow}
\def\ov{\overline}

\def\what{\widehat}
\def\wotimes{\,\what\otimes\,}

\def\Path{{\rm P}}
\def\Tan{{\rm T}}
\def\LHS{{\rm LHS}}
\def\RHS{{\rm RHS}}

\def\PathM{\Path M}

\def\PathMx{\Path(M,x_0,x_1)}

\def\st{\,|\,}
\def\op{^{op}}
\def\Gam{\Gamma}

\newtheorem{thm}{Theorem}[section]
\newtheorem{lem}[thm]{Lemma}
\newtheorem{cor}[thm]{Corollary}
\newtheorem{prop}[thm]{Proposition}

\theoremstyle{definition}
\newtheorem{defn}[thm]{Definition}

\theoremstyle{remark}
\newtheorem{rem}[thm]{Remark}

\numberwithin{equation}{section}

\begin{document}

\title{Iterated integrals of superconnections}

\author{Kiyoshi Igusa}


\begin{abstract} Starting with a $\ZZ$-graded superconnection on a graded vector bundle over a smooth manifold $M$, we show how Chen's iterated integration of such a superconnection over smooth simplices in $M$ gives an $A_\infty$ functor if and only if the superconnection is flat. If the graded bundle is trivial, this gives a twisting cochain.

Very similar results were obtained by K.T. Chen using similar methods. This paper is intended to explain this from scratch beginning with the definition and basic properties of a connection and ending with an exposition of Chen's ``formal connections'' and a brief discussion of how this is related to higher Reidemeister torsion.
\end{abstract}

\address{Department of Mathematics, Brandeis University, Waltham, MA 02454}

\email{igusa@brandeis.edu}






\maketitle

\tableofcontents



\section*{Introduction}\label{sec:intro}

Superconnections on $\ZZ/2$-graded bundles were introduced by Quillen in \cite{Quillensuperconnections}. Later, Bismut and Lott used $\ZZ$-graded superconnections to define analytic torsion forms and analytic torsion classes for smooth bundles \cite{BismutLott95}. The cohomology bundle associated to a smooth manifold bundle $X\to E\to M$, i.e., the bundle over $M$ whose fiber is the deRham cohomology of the fiber $X$ has a flat $\ZZ$-graded superconnection which is unique in a certain sense. By Ed Brown's theorem about twisting cochains \cite{Brown59TwistedTensor}, this superconnection on the fiberwise cohomology bundle contains the information of the differentials in the Serre spectral sequence converging to the cohomology of the total space $E$. Sebastian Goette points out in \cite{Goette01} that these flat superconnections can be viewed as ``infinitesimal twisting cochains.'' And finally, K.T. Chen \cite{Chen73}, \cite{Chen75}, \cite{Chen77} studied ``differential twisting cochains'' which are almost the same as flat superconnections.

The purpose of this paper is to explain from scratch the explicit construction, using Chen's iterated integrals, which produces an $A_\infty$ functor from a flat $\ZZ$-grade superconnection on a graded smooth vector bundle. In the case when the vector bundle is trivial, this gives a twisting cochain. See \cite{ITwisted} for details on the reverse construction: How $A_\infty$ functors and twisting cochains gives rise to flat superconnections. By the definition of an $A_\infty$ functor, this is equivalent to showing that the superconnection parallel transport gives a hierarchy of chain maps, homotopies and higher homotopies if the superconnection is flat. The converse also holds: If the superconnection is not flat we will not get an $A_\infty$ functor. (Corollary \ref{D is flat iff its transport is A-infty})

Here is the idea. It follows from Stokes' theorem that the iterated integral of a superconnection satisfies the boundary condition that makes it a higher homotopy of the one lower degree mappings defined on the boundary if and only if
\[
	d\circ\Psi_k-\Psi_k\circ d=d\Psi_{k-1}.
\]
A sign of $(-1)^k$ will appear when the $k$-form part of $\Psi_k$ goes past the boundary map $d$.
Writing $A_0=d$,
this condition becomes equivalent to the condition:
\[
	dA_{k-1}=\sum_{i=0}^{k}A_iA_{k-i}
\]
which is the definition of a flat superconnection. More precisely $D=d-\sum A_i$ is a flat superconnection if and only $D^2=0$.

After writing the first draft of this paper, I obtained a copy of the complete works of K.T. Chen \cite{Chen} and there I found the same definitions and formulas with very similar notation. Chen calls a flat superconnection a ``differential twisting cochain.'' He integrates flat superconnections over simplices to obtain twisting cochains. The main improvement in the current definition is the variable boundary maps, i.e., $A_0$ in the notation below is a function on $M$. This is crucial for the application to higher Reidemeister torsion. In fact, the higher terms in the superconnection ($A_1,A_2,\cdots$) can be completely ignored because the formula for higher torsion is a polynomial in these terms and a polylogarithm in the $A_0$ term. This is also explained in \cite{ITwisted}. Therefore, we can say that the single aspect of the theory of flat superconnections which is ignored by Chen is the most important for application to higher Reidemeister torsion as developed in \cite{IBookOne}, \cite{IComplexTorsion}.

Very briefly the definitions are as follows, a $\ZZ$-graded superconnection \ref{def: superconnection} on a smooth graded vector bundle $V$ over a smooth manifold $M$ is $(-1)^k$ times an ordinary connection on each graded piece $V^k$ of $V$ together with a sequence of endomorphism valued forms:
\[
	A_p\in\Omega^p(M,\End^{1-p}(V))
\]
of total degree 1 for $p\ge0$. For example,
\[
	A_1=\sum f^i dx_i
\]
where $f^i$ are homogeneous degree zero linear endomorphisms $f_k^i:V^k\to V^k$.

Following Chen \cite{Chen73}, \cite{Chen75}, \cite{Chen77}, we define the parallel transport of a superconnection \ref {def: superconnection parallel transport} to be a sequence of forms on the path space $\PathM$ of $M$ with coefficients in the $\Hom(V_{x_s},V_{x_t})$ bundle (for paths from $\g(s)=x_s$ to $\g(t)=x_t$):
\[
	\Psi_p(\g,t,s)\in\Omega^p(\PathM,\Hom^{-p}(W_s,W_t))
\]
satisfying the differential equation
\[
	{\frac\d{\d t}\Psi_p(t,s)_\g=\sum_{i=0}^p(A_{i+1}/t)\Psi_{p-i}(t,s)}.
\]
This is very easy to solve using integrating factors. The notation $A_k/t$ is the contraction \ref {def: contraction map} of the $k$-form $A_k$ to a $k-1$ form on $\PathM$ by restriction to $t$ and internal multiplication with $\g'(t)$:
\[
	A_k/t:=(-1)^{k-1}\iota_{\g'}ev_t ^\ast A_k(\g)=\iota^R_{\g'}ev_t ^\ast A_k(\g)
\]
To eliminate the inconvenient sign we use the right internal product $\iota^R$.

The first section \ref {section 1: standard definitions} of this paper gives several of the standard definitions of a connection and the parallel transport or holonomy of a connection. The holonomy is defined by a differential equation. The solution of the differential equation is given by an iterated integral \ref {subsec: standard iterated integrals} which in turn can be given by a limit of finite products \ref {subsec: standard product limit} of a matrix 1-form with itself.

Section \ref {sec 2: superconnections} explains K.T. Chen's generalization of these constructions to the superconnection setting. The parallel transport or holonomy of a superconnection is defined by the differential equation mentioned earlier \ref {def: superconnection parallel transport}. The solution of this equation is a sum of iterated integrals \ref {subsec: iterated integral for forms}:
\[
	\iint_{t\ge u_1\ge\cdots\ge  u_n\ge s}du_1\cdots du_n\,
	\Phi(t,u_1)(A_{k_1}/u_1)\Phi(u_1,u_2)(A_{k_2}/u_2)\cdots\Phi(u_n,s)
\]
where $\Phi(t,s)$ is the holonomy of the standard connection from $\g(s)$ to $\g(t)$. This can also be written as a limit of finite products \ref{subsubsec: limit of finite products for superconnections}:
\[
	\Psi(t,s)=\lim_{\max(\Delta_i t)\to0}\prod_{i=1}^k (I+A/u_i\Delta_it)
\]
where $A=\sum A_i$.

In Section \ref{sec 3: integrating flat superconnections} we show that the superconnection parallel transport gives higher homotopies on families of paths parametrized by cubes:
\[
	\g:I^k\to\PathM
\]
Since these are given by smooth maps $I^{k+1}\to M$, we choose a trivialization of the bundle and write the superconnection as $D=d-A$. Over each $j$-face $c^j$ of $I^k$ we integrate the superconnection parallel transport to get a degree $-j$ endomorphism of the fiber $V$. The $k$ cube $I^k$ has $2k$ codimension-one faces. We show that the sum of the degree $1-k$ endomorphisms over these faces, with appropriate signs, is the coboundary of the degree $-k$ endomorphism given by integrating $\Phi_k$ over $I^k$ if and only if the superconnection is flat (Theorem \ref {thm:flat superconnections give higher homotopies}).

In Section \ref {sec 4: A-infty and twisting} we look at simplices. After a general discussion of twisting cochains with variable boundary maps, we construct Chen's mapping \ref{subsec: integration over simplices}
\[
	\theta_{(k)}:I^{k-1}\to\Path(\Delta^k,v_0,v_k)
\]
given by a piecewise linear epimorphism
\[
	\pi_k\circ\ll_k:I^k\to \Delta^k
\]
and show (Theorem \ref {thm:D is flat iff psi is a twisting cochain}) that the integral of a flat superconnection parallel transport along $\theta_{(k)}$ gives a twisting cochain on the subcomplex $C_\ast(M)$ of the singular chain complex of $M$ generated by smooth simplices $\s:\Delta^k\to M$ in the special case when the bundle is trivial, e.g., if we restrict to local coordinates. In general (Corollary \ref {D is flat iff its transport is A-infty}) this construction gives an $A_\infty$ functor, or equivalently, a differential graded functor on the cobar construction.

The mapping $\pi_k\circ\ll_k$ is degenerate on two faces, since it sends $I^{k-1}\times 0$ to the first vertex $v_0$ and $I^{k-1}\times 1$ to the last vertex $v_k$. The other $2k-2$ faces of $I^{k}$ give the faces of $I^{k-1}$. These go to the $k+1$ faces of $\Delta^k$ by maps $\th_{k}$ and the $k-1$ products of faces indicated below with the map $\theta_{(i)}\times \theta_{(j)}$ where $i+j=k$ and $i=1,\cdots,k-1$ (Lemma \ref {positive face lemma}).
\[
	\theta_{(i)}\times \theta_{(j)}:I^{i-1}\times I^{j-1}\to \Path(\Delta^i,v_0,v_i)\times\Path(\Delta^j,v_i,v_k).
\]
There is a redundancy of 2 due to the fact that the $0$ and $k$-faces are counted twice.
This inductive description of the mapping $\theta_{(k)}$ implies that integration along $\theta_{(k)}$ of the superconnection transport of a flat superconnection gives a twisting cochain since the boundary terms are exactly the $2k$ terms in the definition of a twisting cochain on the simplicial set of smooth simplices in $M$ with the usual diagonal map given by front face and back face. There is a redundancy of 2 in the definition of twisting cochain since the 0-face and $k$-face occur twice. This is also explained in detail in the appendix to \cite{Chen77}.

In Subsection \ref{subsec: A-infty functors}, we see what happens globally when we do not have a system of local coordinates. In that case, there is a different chain complex $V_x$ over every point $x\in M$ and, given an ordinary flat connection on $V$, we would make $V_x$ into a functor on the path groupoid of $M$. In general, we don't have a flat connection but we have a flat superconnection in many cases. For example, the fiberwise deRham cohomology bundle has a natural flat superconnection as mentioned earlier. Integration of this flat superconnection gives a $A_\infty$ functor which is a twisting cochain with twisted coefficients (Corollary \ref {D is flat iff its transport is A-infty}). Since the cobar construction gives the universal $A_\infty$ functor, we also get a dg functor defined on the cobar construction (Corollary \ref {dg functor on cobar}).

In Section \ref {sec 5: the work of Chen}, we explain the work of K.T. Chen. The current paper might be regarded as a exposition of his work in the language of superconnections where the boundary map for the coefficient system is variable making these coefficients a bundle of chain complexes over the manifold. Since Chen was mainly interested in the homology of the loop space of the manifold, his paths were usually loops. Therefore the boundary map on the coefficients was fixed in his case.

Finally, in Section \ref {sec 6: higher FR torsion}, we give a very brief explanation of how flat superconnections are used in the Morse theoretic construction of higher Reidemeister torsion.

I would like to thank Danny Ruberman and Ivan Horozov for several helpful conversations about iterated integrals. I would also like to thank Jonathan Block for explaining to me his sign convention for superconnection. I should also thank Florin Dumitrescu for explaining to me his paper \cite{Dumitrescu07} on parallel transport of superconnections on supermanifolds. This paper is based on lectures given at Brandeis University and at the Topology and Analysis in interaction conference at Cortona, Italy in June 2008. I would like to thank the organizers of the Cortona conference for inviting me to such a beautiful and stimulating place.

%
%

\section{Standard definitions}\label{section 1: standard definitions}

First I'll go over one of the standard definitions of a connection and its parallel transport given by an iterated integral.


\subsection{Connection}

Suppose that $M^m$ is a smooth ($C^\infty$) $m$-manifold and $\Omega^0(M)=C^\infty(M)$ is the ring of $C^\infty$-functions
\[
	f:M\to\RR.
\]
Then $\Omega^0$ is a contravariant functor from the category of smooth manifolds and smooth maps to the category of $\RR$-algebras. Also $\Omega^0(M)$ is \emph{locally defined} in the sense that it is the space of sections of a smooth vector bundle over $M$, namely the trivial bundle $M\times\RR$.

If $V\to M$ is a smooth vector bundle then $\Omega^0(M,V)=\Gam V$ denotes the space of smooth sections of $M$. This is an $\Omega^0(M)$-module and we say that an $\Omega^0(M)$-module is \emph{locally defined} if it is isomorphic to $\Gam V$ for some $V$. If we restrict to an open set $U\subseteq M$ over which $V$ is trivial: $V|U\cong U\times\RR^n$, then $\Gam U$ is a trivial $\Omega^0(U)$-module of rank $n$. Therefore, we say that every locally defined $\Omega^0(M)$-module is ``locally free.''

Let $\Omega^1(M)$ be the $\Omega^0(M)$-module of smooth $1$-forms on $M$. This is locally defined since it is the space of sections of the cotangent bundle $\Tan^\ast M$ of $M$. In local coordinates a $1$-form is given by
\[
	\a=a_1dx_1+\cdots a_mdx_m.
\]
where $a_1,\cdots,a_m$ are smooth functions on $U\subseteq M$.

We have the map
\[
	d:\Omega^0(M)\to \Omega^1(M)
\]
given locally by
\[
	df=\frac{\d f}{\d x_1}dx_1+\cdots +\frac{\d f}{\d x_m}dx_m.
\]
The map $d$ is $\RR$-linear and satisfies the \emph{Leibnitz rule}
\[
	[d,f]=df
\]
which is a commutator relation of operations where the operation $f$ is multiplication by $f$:
\[
	[d,f](g):=d(f\cdot g)-f\cdot dg=df\cdot g
\]

If $V\to M$ is any smooth vector bundle,
$\Omega^1(M,V)$ is the $\Omega^0(M)$-module of 1-forms on $M$ with coefficients in $V$. This is canonically a tensor product:
\[
	\Omega^1(M,V)\cong\Gam V\, \what\otimes\,\Omega^1(M)=\Gam V\otimes_{\Omega^0(M)}\Omega^1(M)
\]
where we use $\what\otimes$ to denote tensor product over $\Omega^0(M)=C^\infty(M)$. If we use the correspondence between operations on vector bundles $V,W$ over $M$ and operations on corresponding $\Omega^0(M)$-modules $\Gam V,\Gam W$ given by
\begin{enumerate}
\item $\Gam V\, \what\otimes\,\Gam W\cong \Gam (V\otimes W)
$ and
\item $
	\Hom_{\Omega^0(M)}(\Gam V,\Gam W)\cong \Gam \Hom(V,W)\cong\Hom_M(V,W)
$
\end{enumerate}
we get several other isomorphic expressions:
\[
	\Gam V\, \what\otimes\,\Omega^1(M)\cong \Gam(V\otimes\Tan^\ast M)\cong \Gam \Hom(\Tan M,V)\cong\Hom_M(\Tan M,V)
\]
\[
	\cong \Hom_{\Omega^0(M)}(\Gam \Tan M,\Gam V)
\]

\begin{defn}
A \emph{connection} on $V$ is defined to be a linear map
\[
	\nabla:\Omega^0(M,V)\to\Omega^1(M,V)
\]
which satisfies the Leibnitz rule:
\[
	[\nabla,f]=df
\]
for all $f\in\Omega^0(M)$. In other words,
\[
	\nabla(f s)-f\nabla s=(df)s
\]
for all smooth sections $s$ of $V$.
\end{defn}

If $\nabla,\nabla'$ are two connections on $V$, then their difference commutes with multiplication by $f\in\Omega^0(M)$:
\[
	[\nabla-\nabla',f]=df-df=0.
\]
Therefore $\nabla-\nabla'$ is a homomorphism of $\Omega^0(M)$-modules which is the same as an $\End(V)$-valued 1-form on $M$:
\begin{eqnarray*}
	\Hom_{\Omega^0(M)}(\Gam V,\Omega^1(M,V))&\cong& \Hom_{\Omega^0(M)}(\Gam V,\Gam \Hom(\Tan M,V))
\\
	&\cong& \Gam\Hom(V,\Hom(\Tan M,V))
	\\
	&\cong&\Gam\Hom(\Tan M,\Hom(V,V))
\\
	&\cong& \Omega^1(M,\End(V))
\end{eqnarray*}

If we restrict to a coordinate chart $U\subseteq M$ over which $V$ is trivial: $V|U\cong U\times \RR^n$ then one example of a connection on $V|U$ is given by the differential:
\[
	ds=d(s_1,s_2,\cdots,s_n):=(ds_1,\cdots,ds_n)
\]
The restriction to $U$ of an arbitrary connection on $M$ is therefore given by
\[
	\nabla=d-A
\]
where $A\in \Omega^1(U,\End(V))$ is an $n\times n$ matrix value 1-form. Equivalently, $A=(\a_{ij})$ is an $n\times n$ matrix of 1-forms (over the coordinate chart $U$)
\[
	\a_{ij}=\sum_k a_{ij}^kdx_k\in \Omega^1(U), \quad a^k_{ij}\in \Omega^0(U).
\]
Note that $A$ extends uniquely to a linear map, also called $A$,
\[
	A:\Omega^\ast(M,V)\to \Omega^\ast(M,V)
\]
which commutes with the right action of $\Omega^\ast(M)$ and graded commutes with the left action, action being given by right and left multiplication, i.e.,
\[
	A\b(v \a)=(Av)\b\a=(-1)^{|\a|+|\b|}A(v \a)\b=(-1)^{|\b|}\b A(v\a)
\]
for all $\a,\b\in\Omega^\ast(M)$ and $v\in \Gam V$.


\subsection{Superconnections}

Now suppose that $V$ is a $\ZZ$-graded vector bundle over $M$
\[
	V=\bigoplus_{k\in\ZZ} V^k.
\]
Then, by a \emph{graded connection} on $V$ we mean a sequence of operators 
\[
	\nabla_k:\Omega^0(M,V^k)\to\Omega^1(M,V^k)
\]
satisfying the \emph{graded Leibnitz rule}
\[
	[\nabla_k,f]s=\nabla_k(fs)-f\nabla_k s=(df)s=(-1)^ksdf\in\Gam V^k\wotimes\Omega^1(M).
\]
(In other words, $(-1)^k\nabla_k$ is a connection on $V^k$.)
We write this sequence of operators as an infinite sum
\[
	\nabla=\sum \nabla_k.
\]
Locally, this has the form
$
	\nabla=d-A
$ where
\[
	A\in \Omega^1(M,\End^0(V))
\]
where $\End^p(V)$ denotes the degree $p$ endomorphism bundle of $V$:
\[
	\End^p(V)=\prod_k\Hom(V^k,V^{k+p}).
\]
\begin{defn}[BismutLott]\label{def: superconnection}
A \emph{superconnection} on $V$ is defined to be a linear map
\[
	D:\Gam V\to \prod_{p=0}^m\Omega^p(M,V)
\]
of total degree 1:
\[
	D(\Gam V^k)\subseteq \prod_{p=0}^m\Omega^p(M,V^{k+1-p})
\]
which satisfies the graded Leibnitz rule
\[
	[D,f]=df
\]
for all $f\in\Omega^0(M)$.
\end{defn}

As before, the difference of two superconnections is a homomorphism of $\Omega^0(M)$-modules. Also, any graded connection $\nabla$ on $V$ satisfies this condition. Therefore, superconnections are given by
\[
	D=\nabla-A_0-A_2-A_3-\cdots-A_m
\]
where $\nabla$ is a graded connection on $V$,
\[
	A_p\in\Omega^p(M,\End^{1-p}(V))=\End^{1-p}(V)\wotimes\Omega^p(M)
\]
and we write elements of the product as a sum. Also, as before, $A_p$ extends uniquely to an operator on $\Omega^\ast(M,V)$ which commutes with the right action of $\Omega^\ast(M)$ and graded commutes with the left action.

Note that this decomposition of $D$ is unique since $-A_p$ is the $p$-form component of $D$ for all $p\neq1$.


\subsection{Parallel transport} If $\nabla$ is a connection on a vector bundle $V$ over $M$ and $\g:I\to M$ is a smooth path then by \emph{parallel transport} along $\g$ we mean a family of linear isomorphisms
\[
	\Phi_t:V_{\g(0)}\to V_{\g(t)}
\]
which is ``horizontal'' in the sense that, for each $v\in V_{\g(0)}$, we get a section of $V$ over $\g$ which is tangent to the horizontal distribution defined by $\nabla$.

By a \emph{horizontal distribution} we mean a decomposition of the tangent bundle of $V$ into a vertical and horizontal component at each point. The vertical component is the kernel of the projection $\Tan V\to \Tan M$. The horizontal component is uniquely given by the property that it is tangent to any section $s\in\Gam V$ satisfying
\[
	\nabla s(x)=0.
\]
In local coordinates this says $ds(x)=As(x)$ or
\[
	\frac{\d s_i}{\d x_k}(x)=\sum_j a_{ij}^ks_j(x)
\]
for all $k$ if $A=(\a_{ij})=\left(\sum_ka_{ij}^kdx_k\right)$. This is a linear condition on the 1-jet of $s$. Therefore, it is independent of the choice of local trivialization and defines a section (i.e., right inverse) of the tangent map:
\[
	\Tan_{s(x)}V\to \Tan_xM.
\]
Since $v=s(x)$ can be chosen arbitrarily, this defines a horizontal distribution on $V$. Conversely, the horizontal distribution determines the value of each function $a_{ij}^k$ at every point and, therefore, no two connections give the same distribution.

Suppose that $\g$ is an embedded path and we choose local coordinates so that $x_0=t$ and the other coordinates are constant along $\g$. (We call the latter ``transverse coordinates'' for $\g$.) Then the horizontal condition is:
\[
	\frac{\d s_i}{\d t}=\sum_j a_{ij}^0s_j
\]
If we write $\nabla=d-A$ and $A=\sum A^kdx_k$ where $A^k\in\Gam_U\End(V)$ this condition is
\[
	\boxed{\frac\d{\d t}\Phi_t=A^0(t)\Phi_t.}
\]
This is a first order linear differential equation with initial condition
\[
	\Phi_0=I_n
\]
the $n\times n$ identity matrix.

If $A^0(t)$ were constant, the solution of this equation would be
\[
	\Phi_t=e^{tA^0}=\sum_{k=0}^\infty \frac{t^k}{k!}(A^0)^k
\]
In general the parallel transport map $\Phi_t$ is given by an iterated integral.


\subsection{Chen's iterated integrals I} There are two ways to define iterated integrals: 
\begin{enumerate}
\item as an infinite sum of integrals or
\item as a limit of finite products.
\end{enumerate}
The second definition will be easier to generalize.


\subsubsection{as infinite sum of integrals}\label{subsec: standard iterated integrals}

\begin{defn}
Let $A(t),t\in I$ be a continuous family of $n\times n$ matrices. Then the \emph{iterated integral} of $A$ from $s$ to $t$ is defined by
\[
	\Phi(t,s)=\sum_{k=0}^\infty \iint_{t>t_1>t_2>\cdots>t_k>s}A(t_1)A(t_2)\cdots A(t_k)dt_1dt_2\cdots dt_k
	\]
	\[=I_n+\int_s^t A(t_1)dt_1 +\iint_{t>t_1>t_2>s}A(t_1)A(t_2)dt_1dt_2+\cdots
\]
If $A$ is a continuous matrix 1-form on $U$ and $\g:[s,t]\to U$ is a $C^1$ path then the iterated integral of $A$ along $\g$ is defined to be the iterated integral of the continuous family of matrices
\[
	A^\ast(t)=A_{\g(t)}\g'(t)
\]
where $A_{\g(t)}\g'(t)\in\End(V_{\g(t)})$ is the value of $A_{\g(t)}\in \End(V_{\g(t)})\otimes\Tan^\ast_{\g(t)}U $ on the velocity vector $\g'(t)\in \Tan_{\g(t)}U$.
\end{defn}
Note that since $A(u)$ is bounded by, say $||A(u)||\le B$ for all $t\ge u\ge s$, the $k$th term in the series is bounded by
\[
	\frac{(t-s)^k}{k!}B^k.
\]
So, the series converges absolutely and
\[
	||\Phi(t,s)||\le e^{(t-s)B}.
\]
\begin{lem}\label{Phi is bounded} For any $t\ge u\ge s$ we have the following.
\[
\Phi(t,s)=\Phi(t,u)\Phi(u,s)
\]
\[
\Phi(t,s)=I_n+(t-s)A(u)+o(t-s)=I+A(u)\Delta t+o(\Delta t)
\]
\end{lem}

\begin{proof} The first statement follows from the definition of the iterated integral and the fact that $A(t_j)dt_i=dt_iA(t_j)$.
The second condition follows from the fact that $A(u)$ converges uniformly to $A(t)$ as $s\to t$ and the terms after the first two terms in the series for $\Phi(t,s)$ are bounded by a constant times $(t-s)^2$ since
\[
	e^{(t-s)B}-1-(t-s)B\le (t-s)^2B^2e^{(t-s)B}
\]
if $B$ is a uniform bound for $||A(u)||$.
\end{proof}

\begin{prop}\label{prop:dPhi} $\Phi(t,s)$ satisfies the differential equation (boxed above) and therefore gives the parallel transport of the connection $\nabla=d-A$ on $V$. In fact,
\[
	d\Phi(t,s)=A(t)\Phi(t,s)dt-\Phi(t,s)A(s)ds
\]
\[
	=A_{\g(t)}\Phi(t,s)d\g(t)-\Phi(t,s)A_{\g(s)}d\g(s).
\]
\end{prop}

\begin{proof}
$\Phi(t+\Delta t,s)=\Phi(t+\Delta t,t)\Phi(t,s)=(I+A(t)\Delta t+o(\Delta t))\Phi(t,s)$ by the lemma above. So,
\[
	\Phi(t+\Delta t,s)-\Phi(t,s)=A(t)\Phi(t,s)\Delta t+o(\Delta t).
\]
Similarly,
\[
	\Phi(t,s+\Delta s)(I+A(s)\Delta s+o(\Delta s))=\Phi(t,s).
\]
So,
\[
	\Phi(t,s+\Delta s)-\Phi(t,s)=-\Phi(t,s+\Delta s)A(s)\Delta s+o(\Delta s)
\]
and the proposition follows.
\end{proof}


\subsubsection{as limit of finite products}\label{subsec: standard product limit}
Consider a $C^1$ path $\g:[s,t]\to M$ with image in a coordinate chart $U$. Take all possible partitions of the interval $[s,t]$ into subintervals by choosing $k\ge1$ and
\[
	t=t_0>t_1>t_2>\cdots>t_k=s.
\]
The \emph{mesh} of this partition is the maximum value of $t_{i-1}-t_{i}$. On each subinterval choose a point
\[
	u_i\in[t_i,t_{i-1}].
\]
Then the $i$th increment of $\g$ is
\[
	\Delta_i\g:=\g(t_{i-1})-\g(t_i)=(t_{i-1}-t_i)\g'(u_i)+o(\mesh(t_\ast))
\]
where $\g'(u_i)$ is the velocity vector of $\g$. Note that $o(\mesh(t_\ast))$ goes to zero faster than $1/k$.

\begin{prop}\label{prop:parallel transport as limit of finite products}
Parallel transport by $\nabla=d-A$ is given by
\[
	\Phi(t,s)=\lim_{\mesh(t_\ast)\to0}\prod_{i=1}^k (I+A_{\g(u_i)}\Delta_i\g)
\]
where $A_{\g(u_i)}\Delta_i\g$ is the value of $A_{\g(u_i)}$ on the vector $\Delta_i\g$. The terms in the product are ordered from left to right.
\end{prop}

\begin{proof} By Lemma \ref{Phi is bounded} we know that $\Phi(t,s)$ is given by the product:
\[
	\Phi(t,s)=\prod_{i=1}^k \Phi(t_{i-1},t_i)
\]
and each factor has the estimate
\[
	\Phi(t_{i-1},t_i)=I+(t_{i-1}-t_i)A_{\g(u_i)}\g'(u_i)+o(t_{i-1}-t_i)
\]
\[
	=I+A_{\g(u_i)}\Delta_i\g+o(\mesh(t_\ast)).
\]
Each of the $k$ $o(\mesh(t_\ast))$ terms is multiplied by something which is bounded. Since $o(\mesh(t_\ast))$ goes to zero faster than $1/k$ the sum of these terms goes to zero as $\mesh(t_\ast))\to 0$. The proposition follows.
\end{proof}

\begin{cor}
Parallel transport is invariant under reparametrization of $\g$. In other words, if $\f$ is a diffeomorphism of $I$ then the parallel transport of $\g$ from $\g(\f(s))$ to $\g(\f(t))$ is the same as the parallel transport of $\g\circ\f$ between the same two points of $M$:
\[
	\Phi_\g(\f(t),\f(s))=\Phi_{\g\circ\f}(t,s).
\]
\end{cor}

\begin{proof}
The product formula for $\Phi(t,s)$ with $u_i=t_i$ depends only on the sequence of points $x_i=\g(t_i)\in M$.
\end{proof}

%
%

\section{Parallel transport of a superconnection}\label{sec 2: superconnections}

In a series of papers (\cite{Chen73},\cite{Chen77}, \cite{Chen}) K.T. Chen constructed the parallel transport of a superconnection
\[
	D=\nabla-A_0-A_2-\cdots-A_m
\]
on a graded vector bundle $V=\bigoplus V^k$ over $M$ in the case when $A_0=0$. However, since $A_0$ does not enter into the definition or construction of the parallel transport, this is not a restriction. The parallel transport is a $\Hom$-valued differential form on the space of smooth paths $\g:I\to M$. So, first we need to consider the path space $\PathM$ of $M$. Then we need to consider differential forms on $\PathM$. 

We need to consider smooth families of paths and converging sequences of differential forms at a fixed path. So, we need a ``smooth structure'' on $\PathM$ and a 1st countable topology on the fibers of bundles over $\PathM$. We will consider these two structures separately.


\subsection{The space of paths in $M$}

Let $\PathM$ be the space of all smooth ($C^\infty$) paths $\g:I\to M$ and let $\Path(M,x_0,x_1)$ be the subspace of smooth path $\g$ from $\g(0)=x_0$ to $\g(1)=x_1$. We give these path spaces the $C^1$-topology. If we assume that $M$ is smoothly embedded in Euclidean space, the $C^1$-topology on $\PathM$ is a metric topology given by
\[
	d(\g_0,\g_1)=\max_{t\in I}(||\g_1(t)-\g_0(t)||+||\g_1'(t)-\g_0'(t)||)
\]
I.e., $\g_0,\g_1\in\PathM$ are close if their \emph{$1$-jets} $(\g_i(t),\g_i'(t))$ are close at all $t\in I$.

We also have a ``smooth structure'' on the path space given by the collection of all smooth families of paths. A mapping $f$ from a smooth manifold $U$ into $\PathM $ is called \emph{smooth} if its adjoint
\[
	\what f:U\times I\to M
\]
given by $\what f(u,t)=f(u)(t)$
is smooth. Chen called such mappings ``plots'' and developed their properties axiomatically. We will not need this level of generality. The collections of smooth manifolds over $\PathM$ forms a category where a morphism $(U,f)\to (V,g)$ is a smooth map $h:U\to V$ so that $gh=f$. Smooth mappings into $\PathM$ are clearly continuous in the $C^1$-topology.

\def\ff{\varphi}

One particularly useful family of deformations of any smooth path $\g$ is given as follows. The pull-back of the tangent bundle of $M$ along $\g$ is an $m$-plane bundle over $I$ which is trivial since $I$ is contractible. Thus $\g:I\to M$ is covered by a map of vector bundles $\tilde\g:I\times\RR^m\to \Tan M$. If we let $U$ be an open $\e$-ball around 0 in $\RR^m$ and we choosing a Riemannian metric and exponential map for $M$, we get a smooth map $\ff:I\times U \to M$ by
\[
	\ff(t,u)=exp_{\g(t)}\tilde\g(t,u).
\]
By construction, this smooth map will have the property that it takes $t\times U$ diffeomorphically onto an open neighborhood of $\g(t)$ in $M$. This means that any small deformation $\eta$ of $\g$ can be written as the composition of $\ff$ with an embedding $t\mapsto(t,\tilde\eta(t))$:
\[
	\eta(t)=\ff(t,\tilde\eta(t))
\]
where $\tilde\eta:I\to U$ is uniquely determined. Thus, we call $\ff$ a \emph{universal deformation} of $\g$. By factoring $\g$ through a universal deformation, we can usually assume that $\g$ is an embedding and that we have a system of local coordinates in a neighborhood of the image of $\g$.

We will insist that our constructions on the path space are not only continuous but smooth in the sense that they are smooth on smooth families of path. For example, we have the following.

\begin{defn}
A mapping $f:\PathM\to\RR$ will be called \emph{smooth} if it is continuous in the $C^1$-topology and if, for any smooth mapping $\f:U\to \PathM$, the composition $f\circ\f:U\to\RR$ is a smooth function. The space of all smooth real valued functions on $\PathM$ will be denoted $\Omega^0(\PathM)$.
\end{defn}


\subsection{Bundles over the path space}

We want to consider the tangent bundle $\Tan\PathM$, cotangent bundle $\Tan^\ast\PathM$ and differential forms $\Omega^p(\PathM)$ on the path space.

\begin{defn}
The \emph{tangent bundle} $\Tan\PathM $ of $\PathM $ is defined to be the vector bundle whose fiber over $\g\in\PathM $ is the vector space of all smooth sections $\eta:I\to \Tan M$ of the tangent bundle of $M$ along $\g$. We call these \emph{vector fields along $\g$}. Since $\eta$ determines $\g$ we can topologize $\Tan\PathM$ as a subspace of the path space of the tangent bundle of $M$ with the $C^1$ path space topology.
\end{defn}

The space $\Tan_\g\PathM$ of vector fields along $\g$ is a topological vector space whose topology, defined above, can also be given by an equivalence class of norms as we now explain.


\subsubsection{norm topology}

The algebra $\LL=\Omega^0(I)$ of all smooth maps $f:I\to \RR$ has a norm given by
\[
	||f||:=\max_{t\in I} (|f(t)|+|f'(t)|).
\]
By the product rule, this norm has the property that
\[
	||f\cdot g||\le ||f||\cdot||g||
\]
for all $f,g\in\LL$. The free module $\LL^n$ has a norm:
\[
	||F||:=\max_{t\in I}(||F(t)||+||F'(t)||)
\]
where $||F(t)||,||F'(t)||$ are the Euclidean norms. E.g.,
\[
	||F(t)||=\sqrt{\sum F_i(t)^2}.
\]
This norm has the property that
\[
	||gF||\le ||g||\cdot||F||
\]
for all $g\in\LL,F\in\LL^n$. Consequently:
\begin{lem}
Any $\LL$-linear homomorphism $\LL^n\to \LL^m$ is bounded in this norm.
\end{lem}

We note that a $\LL$-linear map $\LL^n\to \LL^m$ is given by a smooth family of $m\times n$ real matrices parametrized by $I$ since a matrix of smooth functions on $I$ is the same as a smooth matrix function on $I$. This implies the following.

\begin{prop}
Any finitely generated free $\LL$ module has a unique 1st countable topology with the property that any isomorphism with $\LL^n$ is a homeomorphism.
\end{prop}

\begin{prop}The tangent space $\Tan_\g\PathM$ at $\g$ is a free $\LL$-module of rank $m=\dim M$ and the norm topology is equivalent to the $C^1$ path space topology.
\end{prop}

We recall that a \emph{bounded linear functional} on $\LL^n$ is an $\RR$-linear mapping $\f:\LL^n\to \RR$ so that the ratio $|\f(f)|/||f||$ is bounded for all nonzero $f\in\LL^n$. The least upper bound of this ratio is denoted $||\f||$. Linear functionals on $\LL^n$ are continuous if and only if they are bounded. We will only consider bounded functionals. Thus:
\[
	\Hom(\LL^n,\RR):=\{\f:\LL^n\to\RR\st\text{$\f$ is linear with } ||\f||<\infty\}.
\]

We define the \emph{cotangent space} $\Tan_\g^\ast\PathM$ of $\PathM$ at $\g$ to be the space of continuous linear functionals on the tangent space $\Tan_\g\PathM$:
\[
	\Tan_\g^\ast\PathM=\Hom(\Tan_\g\PathM,\RR).
\] 
We call these ``bounded'' instead of ``continuous'' since ``continuous and smooth'' sounds redundant.

\subsubsection{vector fields and derivatives}

Note that for any smooth manifold $(U,\f)$ over $\PathM$, the map $\f$ is covered by a morphism of vector bundles $\Tan \f:\Tan U\to \Tan\PathM$ given by taking a tangent vector $v$ in $U$ to the vector field
\[
	\eta(t)=\frac\d{\d u}\what\f(uv,t)|_{u=0}.
\]

\begin{prop} Given a smooth function $f:\PathM\to\RR$ and any path $\g_0$ in $M$, there is a unique bounded linear functional
\[
	D_{\g_0}f: \Tan_{\g_0}\PathM\to \RR
\]
(which we also write as $D_{\g_0}f(\eta)=D_\eta f(\g_0)$ and call the derivative of $f$ at $\g_0$ in the direction of $\eta$) satisfying the chain rule:
\[
	D_{\g_0}f(D_0(\g_u))=D_0(f\circ\g_u):I\to \RR
\]
for any smooth deformation $\g_u$, $u\in(-\e,\e)$, of $\g_0$.
\end{prop}

\begin{proof}
To show that $D_{\g_0}f$ exists, we take, for any $\eta\in \Tan_{\g_0}\PathM$, a one parameter deformation $\g_u$ of $\g_0$ so that $D_0\g_u=\eta$. Then we simply define $D_{\g_0}f(\eta)$ to be $D_0(f\circ\g_u)$. We need to show that this is well defined and linear in $\eta$. Linearity is easy. To show it is well defined it suffices (by linearity) to consider the case $\eta=0$.

Working inside a universal deformation we may assume that $M=\RR^m$. Then any one parameter deformation $\g_u$ of $\g_0$ with $D_0\g_u=\eta=0$ can be written as
\[
	\g_u=\g_0+u^2\psi(u)
\]
where
\[
	\psi(u)=\iint_{0\le x\le y\le1}D_0^2\g_{ux}\,dxdy.
\]
Therefore,
\[
	D_0(f\circ\g_u)=D_0(f(\g_0+u^2\psi(u))|_{u=0}
\]
\[
	=\frac\d{\d u}D_0(f(\g_0+u^2\psi(v))|_{(u,v)=(0,0)}
	+\frac\d{\d v}D_0(f(\g_0+u^2\psi(v))|_{(u,v)=(0,0)}
\]
\[
	=0+0=0.
\]
This shows that the derivative $D{\g_0}f$ of $f$ at $\g_0$ exists. The fact that it is bounded follows from the continuity of $f$.
\end{proof}

The derivative that we just defined gives a linear map
\[
	d:\Omega^0(\PathM)\to\Omega^1(\PathM)
\]
Where $\Omega^0(\PathM)$ is the algebra of smooth (and bounded) real valued function on $\PathM$ and the space $\Omega^p(\PathM )$ of smooth $p$-forms on $\PathM $ is defined as follows. 


\subsubsection{smooth $p$-forms on the path space}

A \emph{$p$-form} on $\PathM$ is a function which assigns to each $\g\in\PathM $ an alternating $p$-form on $\Tan_\g\PathM $. A $p$-form is called \emph{smooth} if its pull-back to any smooth manifold $U$ over $\PathM$ is a (smooth) $p$-form on $U$ and if it is continuous on the subspace of $(\Tan\PathM)^p$ on which it is defined.

We will be taking limits of sequence of forms on the path space. These sequences will be locally uniformly converging in norm and will therefore have well defined continuous limits.

For every $t\in I$ we have the evaluation map
\[
	ev_t:\PathM \to M
\]
sending $\g$ to $\g(t)$. This is a smooth map in the sense that it is $C^1$-continuous and, for any smooth manifold $U$ over $\PathM$, the composition $U\to \PathM\to M$ is a smooth map. In the future we will leave it to the reader to figure out what ``smooth'' means in reference to various structures on $\PathM$.

There is an induced smooth map on tangent bundles 
\[
	\Tan_\g ev_t:\Tan_\g \PathM \to \Tan_{\g(t)}M
\]
which is also given by evaluation at $t$:
\[
	\Tan_\g ev_t(\eta)=\eta(t).
\]
Let $W_t$ be the pull back of $V$ to $\PathM $ along $ev_t$. Thus
\[
	(W_t)_\g=V_{\g(t)}.
\]
$W_t$ is a smooth $\ZZ$-graded bundle over $\PathM$. The grading is given by
\[
	W_t=\bigoplus W_t^k=\bigoplus ev_t^\ast V^k.
\]

If $0\le s\le t\le 1$ let
\[
	\Hom^q(W_s,W_t)
\]
be the smooth bundle over $\PathM$ of degree $q$ graded homomorphisms from $W_s$ to $W_t$. Let 
\[
	\Omega^p(\PathM ,\Hom^q(W_s,W_t))=\Hom^q(W_s,W_t)\otimes_{\Omega^0(\PathM)}\Omega^p(\PathM)
\]
be the smooth bundle over $\PathM $ whose fiber over $\g$ is the vector space of smooth $p$-forms on $\Tan_\g \PathM $ with coefficients in 
\[
	\Hom^q(W_s,W_t)_\g=\Hom^q(V_{\g(s)},V_{\g(t)}).
\]

\begin{defn}\label{def: contraction map}
We define the \emph{contraction map}
\[
	/t:\Omega^{p+1}(M,\End^q(V))\to \Omega^{p}(\PathM ,\End^q(W_t))
\]
to be the linear mapping which sends a $p+1$-form $\a$ on $M$ with coefficients in $\End^q(V)$ to the smooth $p$-form on $\PathM $ whose value at $\g$ is the alternating map
\[
	(\a/t)_\g:(\Tan_\g\PathM )^p\to\End^q(W_t)_\g=\End^q(V_{\g(t)})
\]
given by
\[
	(\a/t)_\g(\eta_1,\eta_2,\cdots,\eta_p)=\a(\eta_1(t),\cdots,\eta_p(t),\g'(t)).
\]
Up to sign, $(\a/t)_\g$ is the (right) interior product of $\g'$ on $(ev_t^\ast\a)_\g$ where $(ev_t^\ast\a)_\g\in\Omega^{p+1}_\g\PathM$ is the pull-back of $\a$ along the evaluation map $ev_t:\PathM\to M$:
\[
	(\a/t)_\g=\iota^R_{\g'}(ev_t^\ast\a)_\g=(-1)^p\iota_{\g'}(ev_t^\ast\a)_\g.
\]
\end{defn}
Note that if $p=0$, the contraction $\a/t$ at $\g$ is simply the evaluation of the 1-form $\a$ at the point $\g(t)$ on the vector $\g'(t)$.


\subsection{Parallel transport}

With this notation, parallel transport of a connection $\nabla$ is the unique family of sections
\[
	\Phi(t,s)\in \Omega^0(\PathM ,\Hom^0(W_s,W_t))
\]
for all $1\ge t\ge s\ge0$ so that, for each path $\g\in\PathM $,
\begin{enumerate}
\item $\Phi(s,s)_\g$ is the identity map in
\[
\Omega^0_\g(\PathM ,\Hom^0(W_s,W_s))=\Hom^0(V_{\g(s)},V_{\g(s)})
\]
\item $\Phi(t,s)_\g$ satisfies the first order linear differential equation
\[
	\frac\d{\d t}\Phi(t,s)_\g=(A/t)_\g\Phi(t,s)_\g=(A_{\g(t)}\g'(t))\Phi(t,s)_\g
\]
where $A$ is given by the decomposition $\nabla=d-A$ in some coordinate chart $U$ for $M$ in a neighborhood of $\g(t)$ and differentiation with respect to $t$ is defined using the chosen product structure on the bundle $V$ over $U$.
\end{enumerate}
\begin{defn}\label{def: superconnection parallel transport}
We define \emph{parallel transport} of the superconnection $D$ to be the unique family of forms $\Psi_p(t,s)$ for all $1\ge t\ge s\ge0$ and $p\ge0$
\[
	\Psi_p(t,s)\in \Omega^p(\PathM ,\Hom^{-p}(W_s,W_t))
\]
satisfying the following at each $\g\in\PathM $.
\begin{enumerate}
\item $\Psi_0(s,s)_\g$ is the identity map in
\[
\Omega^0_\g(\PathM ,\Hom^0(W_s,W_s))=\Hom^0(V_{\g(s)},V_{\g(s)})
\]
and $\Psi_p(s,s)=0$ for $p>0$.
\item For all $p\ge0$ we have
\[
	\boxed{\frac\d{\d t}\Psi_p(t,s)_\g=\sum_{i=0}^p(A_{i+1}/t)_\g\Psi_{p-i}(t,s)_\g}
\]
where $D=d-A_0-A_1-\cdots-A_m$ is the decomposition of $D$ in a coordinate chart $U$ for $M$ in a neighborhood of $\g(t)$ and differentiation with respect to $t$ is given by the chosen product structure on $V|U$.
\end{enumerate}
The differential equation determines $\Psi_p(t,s)$ uniquely as a $p$-form on $\PathM$. We need to show that it is smooth.
\end{defn}

Comparing definitions we see immediately that
\[
	\Psi_0(t,s)=\Phi(t,s).
\]
For $p\ge1$ we can find $\Psi_p(t,s)$ by induction on $p$ using integrating factors. To do this, first note that the differential equation (in the box above) has the form
\[
	\frac\d{\d t}\Psi_p(t,s)=(A_1\g'(t))\Psi_p(t,s)+f(t,s)
\]
where $f(t,s)$ is given in terms of $\Psi_q(t,s)$ for $q<p$:
\[
	f(t,s)=\sum_{q=0}^{p-1} (A_{p-q+1}/t)\Psi_q(t,s).
\]
We multiply by the integrating factor $\Phi(s,t)=\Phi(t,s)^{-1}$ to get
\[
	\frac\d{\d t}\left(\Phi(s,t)\Psi_p(t,s)\right)=\Phi(s,t)\frac\d{\d t}\Psi_p(t,s)-\Phi(s,t)(A_1\g'(t))\Psi_p(t,s)=\Phi(s,t)f(t,s).
\]
So,
\[
	\Phi(s,t)\Psi_p(t,s)=\int_s^t du\,\Phi(s,u)f(u,s).
\]
Multiplying by $\Phi(t,s)$ we get the following.

\begin{prop}\label{superparallel transport formula 1} In degree $0$ we have $\Psi_0(t,s)=\Phi(t,s)$, i.e., parallel transport of the superconnection agrees with that of the underlying connection in degree $0$. For $p\ge1$, $\Psi_p(t,s)$ is given recursively by
\[
	\Psi_p(t,s)=\sum_{q=0}^{p-1}\int_s^t du\,\Phi(t,u)(A_{p-q+1}/u)\Psi_q(u,s)
\]
if $D=d-\sum A_i$ on a coordinate chart containing $\g[s,t]$.
\end{prop}

We note that $du$ commutes with all other factors in the integrand above since they are all even.
Taking a coordinate chart for a universal deformation we see by induction on $p$ that $\Psi_p(t,s)$ is a smooth $p$-form on $\PathM$.

\begin{cor}\label{cor:factorization of superiterated integral}
For all $t\ge u\ge s$ we have:
\[
	\Psi_p(t,s)=\sum_{q=0}^p\Psi_{p-q}(t,u)\Psi_q(u,s).
\]
\end{cor}

\begin{rem}\label{rem:factorization of superiterated integral}
If we write $\Psi(t,s)=\sum_{p\ge0}\Psi_p(t,s)$ we can write this more cleanly as
\[
	\Psi(t,s)=\Psi(t,u)\Psi(u,s).
\]
By induction we also get:
\[
	\Psi(t,s)=\Psi(t,t_1)\Psi(t_1,t_2)\cdots\Psi(t_n,s)
\]
for any $t>t_1>t_2>\cdots>t_n>s$.
\end{rem}

\begin{proof} The proof is by induction on $p$. We already know that the statement holds for $p=0$ since $\Psi_0=\Phi$.
If we break the interval of integration in the proposition into two intervals at the point $u$ we get an expression for $\Psi_p(t,s)$ as a sum of two integrals. The first integral is
\[
	\sum_{q=0}^{p-1}\int_s^u dw\,\Phi(t,w)(A_{p-q+1}/w)\Psi_q(w,s)
\]
\[
	=\Phi(t,u)\sum_{q=0}^{p-1}\int_s^u dw\,\Phi(u,w)(A_{p-q+1}/w)\Psi_q(w,s)=\Psi_0(t,u)\Psi_p(u,s).
\]
The second integral is
\[
	\sum_{q=0}^{p-1}\int_u^t dw\,\Phi(t,w)(A_{p-q+1}/w)\Psi_q(w,s).
\]
Since $q<p$, we can, by induction, write this as
\[
	\sum_{q<p}\sum_{j+k=q}\int_u^t dw\,\Phi(t,w)(A_{p-q+1}/w)\Psi_j(w,u)\Psi_k(u,s).
\]
Rearranging the double sum as $\sum_{k<p}\sum_{j<p-k}$, this becomes:
\[
	\sum_{k=0}^{p-1}\Psi_{p-k}(t,u)\Psi_k(u,s).
\]
Adding the two integrals gives the result.
\end{proof}

\begin{cor}\label{cor:infinitesmal superparallel transport}
For any $t\ge u\ge s$ and $p\ge1$ we have
\[
	\Psi_p(t,s)=(t-s)(A_{p+1}/u)+o(t-s)
\]
where $A_{p+1}/u$ is the $p$ form on the space of paths $\g:[s,t]\to M$ given by contraction at $u\in [s,t]$.
If $\g$ is an embedding and we take local coordinates we can modify this to get
\[
	\Psi_p(t,s)=\iota^R_{\g(t)-\g(s)}A_{p+1}(\g(s))+o(||\g(t)-\g(s)||).
\]
\end{cor}

\begin{proof} The $p=0$ case was excluded since we already know that
\[
	\Psi_0(t,s)=\Phi(t,s)=I+(t-s)(A_1/u)+o(t-s)=I+O(t-s).
\]
For $p=1$ we have
\[
	\Psi_1(t,s)=\int_s^t du\, \Phi(t,u)(A_2/u)\Phi(u,s)
\]
\[
	=\int_s^t du\, (I+O(t-u))(A_2/u)(I+O(u-s)).
\]
If we integrate something which is bounded by a constant times $t-s$ we get something on the order of $(t-s)^2$. Therefore, this is
\[
	\Psi_1(t,s)=(A_2/u)(t-s)+O((t-s)^2)
\]
for any $u\in[s,t]$. For $p\ge2$ we have
\[
	\Psi_p(t,s)=\sum_{k=1}^p\int_s^t du\, (I+O(t-u))(A_{k+1}/u)\Psi_{p-k}(u,s).
\]
If $p>k$ the term $\Psi_{p-k}(u,s)$ is bounded by a constant times $u-s$. Therefore, it contributes a term of order $(t-s)^2$ after integrating. So, only the $k=p$ term counts. But, $\Psi_0(u,s)=I+O(u-s)$. So, 
\[
	\Psi_p(t,s)=O((t-s)^2)+\int_s^t du\, A_{p+1}/u=(t-s)A_{p+1}/u+O((t-s)^2)
\]
for any $t\ge u\ge s$.

The modification on local coordinates follows from the definition of contraction:
\[
	(t-s)A_{p+1}/s=(t-s)\iota^R_{\g'}ev_s^\ast A_{p+1}(\g)=\iota^R_{(t-s)\g'}ev_s^\ast A_{p+1}(\g)
\]
\[
	=\iota^R_{\g(t)-\g(s)}A_{p+1}(\g(s))+o(t-s).
\]
And $t-s$ has the same order as $||\g(t)-\g(s)||$ if $\g$ is an embedding.
\end{proof}

\begin{cor} For $p\ge0$ we have:
\[
	\frac\d{\d s}\Psi_p(t,s)=\sum_{j=0}^{p}-\Psi_{p-j}(t,s)A_{j+1}/s.
\]
\end{cor}

\begin{proof}
Using the above two corollaries we have:
\[
	\Psi_p(t,s)=\sum_{j=0}^p\Psi_{p-j}(t,s+\Delta s)\Psi_j(s+\Delta s,s)
\]
\[
	=\Psi_p(t,s+\Delta s)+\Delta s\sum_{j=0}^p \Psi_{p-j}(t,s+\Delta s)A_{j+1}/s+o(\Delta s).
\]
The statement follows.
\end{proof}


\subsection{Chen's iterated integrals II}\label{subsec: iterated integral for forms}

We will expand the expression for the parallel transport operator given in Proposition \ref{superparallel transport formula 1} to give an iterated integral. We will also construct the parallel transport as a limit of finite products. The latter formula shows that the superconnection parallel transport is independent of the parametrization of the path. As pointed out earlier, these definitions and results are all due to Chen. We explained this more in detail in the last section.


\subsubsection{expansion of the formula}

For $p=1$ we have only one term in the sum:
\[
	\Psi_1(t,s)=\int_s^t du\,\Phi(t,u)(A_2/u)\Phi(u,s).
\]

For $p=2$ we get two summands. Inserting the integral above into the second summand we get:
\[
	\Psi_2(t,s)=\int_s^t du\, \Phi(t,u)(A_3/u)\Phi(u,s)\]
	\[+\iint_{t\ge u_1\ge u_2\ge s}du_1du_2\,
	\Phi(t,u_1)(A_2/u_1)\Phi(u_1,u_2)(A_2/u_2)\Phi(u_2,s).
\]

Continuing in this way, we get the following.

\begin{prop}\label{prop: first iterated integral for Psi} For $p\ge1$, $\Psi_p(t,s)$ is equal to the sum over all $n>0$ and all sequences of integers $k_1,\cdots,k_n\ge2$ so that $\sum k_i=p+n$ of the following iterated integral.
\[
	\iint_{t\ge u_1\ge\cdots\ge  u_n\ge s}du_1\cdots du_n\,
	\Phi(t,u_1)(A_{k_1}/u_1)\Phi(u_1,u_2)(A_{k_2}/u_2)\cdots\Phi(u_n,s)
\]
if $D=d-\sum A_i$ on a coordinate chart containing $\g[s,t]$.
\end{prop}

Substituting the iterated integral expression for $\Phi(u_i,u_{i+1})$ we get the following.
\begin{cor}\label{cor: second iterated integral for Psi} For $p\ge0$,
 $\Psi_p(t,s)$ is equal to the sum over all $n\ge0$ and all sequences of integers $k_1,\cdots,k_n\ge1$ so that $\sum k_i=p+n$ of the following iterated integral.
\[
	\iint_{t\ge u_1\ge\cdots\ge  u_n\ge s}du_1\cdots du_n\,
	(A_{k_1}/u_1)(A_{k_2}/u_2)\cdots(A_{k_n}/u_n)
\]
if $D=d-\sum A_i$ on a coordinate chart containing $\g[s,t]$.
\end{cor}


\subsubsection{a limit of finite products}\label{subsubsec: limit of finite products for superconnections}
Take all possible partitions of the interval $[s,t]$ into subintervals by choosing $k\ge1$ and
\[
	t=t_0>t_1>t_2>\cdots>t_k=s.
\]
The \emph{mesh} of this partition is the maximum value of $\Delta_it=t_{i-1}-t_{i}$. On each subinterval choose a point $u_i\in [t_{i-1},t_i]$.

\begin{prop} If we have local coordinates,
parallel transport by $D=d-A_0-A_1-A_2-\cdots-A_m$ is given by
\[
	\Psi(t,s)=\sum_{p=0}^\infty\Psi_p(t,s)=\lim_{\mesh(t_\ast)\to0}\prod_{i=1}^k (I+A/u_i\Delta_it)
\]
where $A=A_1+A_2+\cdots+A_m$. The terms in the product are ordered from left to right.
\end{prop}

\begin{proof} By Remark \ref{rem:factorization of superiterated integral} we can write $\Psi(t,s)$ as a product:
\[
	\Psi(t,s)=\prod \Psi(t_{i-1},t_i)
\]
By Corollary \ref{cor:infinitesmal superparallel transport}, each factor has an approximation:
\[
	\Psi(t_{i-1},t_i)=I+A/t_i\Delta_it+o(\mesh)
\]
The proposition follows, just as in Proposition \ref{prop:parallel transport as limit of finite products}.
\end{proof}

We have the following corollary which is more difficult to state than to prove.

\begin{cor}\label{invariance under reparametrization}
The superconnection parallel transport $\Psi_p(t,s)$ is independent of parametrization in the following sense. Suppose that $\g_u(t),u\in U$ is any smooth family of paths in $M$ and $\f_u:I\to I,u\in U$ is any smooth family of smooth homeomorphisms of $I$ with the property that $\f_u(t)=t$ and $\f_u(s)=s$ for some fixed $1\ge t>s\ge 0$. Then, the parallel transport of $D$ along $\g_u\circ\f_u$ from $s$ to $t$ as a $p$-form on $U$ at $u$ is equal to the parallel transport of $D$ along $\g_u$ between the same two points:
\[
	\Psi_p(t,s)_{\g_u\f_u}=\Psi_p(t,s)_{\g_u}\in\Omega_u^p(U,\Hom(W_s,W_t)).
\]
\end{cor}

\begin{proof}
The product formula can be written in the following parametrization independent way:
\[
	\Psi(t,s)=\lim\prod (I+\iota^R_{x_{i-1}-x_i}A(x_i))
\]
where $x_i=\g(t_i)$. If we replace $x_i$ with $y_i=\g(\f(t_i))$ the distance between points will grow by a bounded factor. So this limit will still converge to the same thing.
We use the fact that the reparametrization $\f_u$ preserves the transverse coordinates (the coordinates of $U$).
\end{proof}

%
%

\section{Integration of flat superconnections}\label{sec 3: integrating flat superconnections}

In this section we will make a list of conditions (in boxed equations) that we would like the matrix forms $A_p$ to satisfy. We will see that these conditions hold if and only if $D=\nabla-A_0-A_2-A_3-\cdots$ is a flat superconnection in the sense of Bismut and Lott.

Let's start with $p=0$. So far we haven't used $A_0\in\Omega^0(M,\End^1(V))$. We want $(V,A_0)$ to be a cochain complex bundle over $M$. For this we just need:
\[
	\boxed{A_0^2=0.}
\]
Next, we want the parallel transport $\Phi$ to be a cochain map
\[
	\Phi(t,s):V_{\g(s)}\to V_{\g(t)}.
\]
This will hold if $A_1$ is suitably chosen.
If the connection $\nabla$ is not flat, this will depend on the path $\g$. But, for suitable $A_2$, integration of $\Psi_1$ will give a homotopy between chain maps. Similarly, the integral of $\Psi_2$ will give a homotopy between these homotopies if $A_3$ is suitably chosen. And so on. The sequence of conditions that we get on the $A_p$ make $D=\nabla-A_0-A_2-A_3-\cdots$ into a flat superconnection.


\subsection{Parallel transport as cochain map}

For any path $\g$ in $M$ we want the parallel transport $\Phi$ of $\nabla=d-A_1$ to be a cochain map as indicated above. In other words, we want:
\[
	\boxed{A_0(\g(t))\Phi(t,s)=\Phi(t,s)A_0(\g(s))}
\]
for all $1\ge t\ge s\ge 0$.
This is equivalent to the condition that $\Phi(t,u)A_0(\g(u))\Phi(u,s)$ is independent of $u\in[s,t]$. This, in turn, is equivalent to the condition:
\[
	\frac\d{\d u}\Phi(t,u)A_0(\g(u))\Phi(u,s)=0.
\]
Taking local coordinates for $M$ at the point $\g(u)$ we can write this as:
\[
	\Phi(t,u)\left[
	-A_1/u\,A_0-dA_0/u+A_0\,A_1/u
	\right]\Phi(u,s)=0.
\]
Here we used the fact that, since $A_0$ is odd,
\[
	\frac\d{\d u}A_0(\g(u))=-dA_0/u.
\]
Since $\Phi(t,u),\Phi(u,s)$ are invertible, this is equivalent to the condition:
\[
	-A_1/u\,A_0-dA_0/u+A_0\,A_1/u=0.
\]
Using again the fact that $A_0$ is odd we see that $A_1/u\,A_0=-(A_1A_0)/u$. Therefore, if we want this to hold for all paths $\g$ we must have:
\[
	\boxed{dA_0=A_0A_1+A_1A_0.}
\]
\begin{thm} Suppose that $A_0^2=0$ at all points in $M$. Then, parallel transport $\Phi(t,s)$ is a chain map for all paths $\g$ in $M$ if and only ${dA_0=A_0A_1+A_1A_0}$ at all points of $M$.
\end{thm}


\subsection{Cochain homotopy} Let $x_0,x_1\in M$ and suppose that $\g_0,\g_1$ are paths from $x_0$ to $x_1$. Then we get two cochain maps
\[
	V_{x_0}\to V_{x_1}.
\]
If $\nabla$ is flat and $\g_0,\g_1$ are homotopic fixing their endpoints, these two cochain maps will agree. If $\nabla$ is not flat, we would like to have a cochain homotopy between these two cochain maps. So, suppose we have a homotopy of paths
\[
	h:I_1\times I_2\to M
\]
which we write as a 1-parameter family of paths $\g_u(t), t\in I_1,u\in I_2$ so that $\g_u(0)=x_0$ and $\g_u(1)=x_1$ for all $u\in I_2$. (Here $I_1=I_2=[0,1]$. The indices are just to keep track of the coordinates.)

For each $u\in I_2$ we have a cochain map:
\[
	\Phi_u(1,0):W_0=V_{x_0}\to W_1=V_{x_1}
\]
given by parallel transport of $\nabla$ along the path $\g_u$.

If we pull back $V,A_0,A_1,A_2$ to $I^2$ along $h$ we also get the pull back of the superconnection transport:
\[
	\Psi_1(1,0)\in \Omega^1(I_2,\Hom^{-1}(W_0,W_1)).
\]
We want the integral of this 1-form on $I_2$ to be a homotopy $\Phi_0\simeq \Phi_1$. In order for this to hold for all homotopies of all paths we must have:
\[
	A_0\int_{u_1}^{u_2}\Psi_1+\left(\int_{u_1}^{u_2}\Psi_1\right)A_0=\Phi_{u_2}-\Phi_{u_1}.
\]
This is equivalent to its infinitesimal form which has a different sign:
\[
	\boxed{A_0\Psi_1-\Psi_1A_0=d\Phi} 
\]
The negative sign in the differential form becomes positive when we collect terms:
\[
	(\d \otimes 1)\left(\sum h_i\otimes dx_i\right)-\left(\sum h_i\otimes dx_i\right)(\d\otimes 1)=\sum (\d h_i+h_i\d)dx_i.
\]
If we integrate $\Psi_1$ first and then multiply by $A_0$ there is no $dx_i$ term so the sign does not change. On the RHS we view $\Phi$ as a 0-form on $I_2$:
\[
	\Phi\in\Omega^0(I_2,\Hom^0(W_0,W_1)).
\]
Recall that
\[
	\Psi_1(1,0)=\int_0^1dt\, \Phi(1,t)(A_2/t)\Phi(t,0).
\]
Since $\Phi$ is a cochain map, the LHS of the equation in the box is
\[
	A_0\Psi_1(1,0)-\Psi_1(1,0)A_0=\int_0^1dt\,\Phi_u(1,t)[\underbrace{A_0A_2/t-A_2/tA_0}_{(A_0A_2+A_2A_0)/t}]\Phi_u(t,0).
\]
The RHS of the boxed equation is
\[
	d\Phi_u(1,0)=\int_0^1dt\, \Phi_u(1,t)d(A_1/t)\Phi_u(t,0).
\]
\begin{lem}\label{first lemma} 
In $I^n$ we have:
\[
	d(A_k/t)=\frac\d{\d t}A_k^\perp+(dA_k)/t
\]
where $A_k^\perp=A_k-(A_k/t)dt$ is the part of $A_k$ which involves only $dx_i$ for $i\ge2$.
\end{lem}

Using this lemma for $k=1$, the condition that we need becomes:
\[
	\int_0^1dt\,\Phi_u(1,t)\left
	[(A_0A_2+A_2A_0)/t-(dA_1)/t-\frac\d{\d t}A_1^\perp
	\right]\Phi_u(t,0)=0.
\]
Consider the one parameter family of functions
\[
	f_u(t)=\Phi_u(1,t)A_1^\perp(t,u)\Phi_u(t,0),\quad u\in I.
\]
Since the paths $\g_u$ are fixed at the two ends, we have
\[
	A_1^\perp(0,u)=0=A_1^\perp(1,u)
\]
for all $u$. Therefore, $f_u(0)=0=f_u(1)$ for all $u$. This means that the integral of the derivative of $f_u$ is zero:
\[
	\int_0^1dt\,\frac\d{\d t}\Phi_u(1,t)A_1^\perp(t,u)\Phi_u(t,0)=0
\]
\[
	\int_0^1dt\, \Phi_u(1,t)\left[
	-(A_1/t)A_1^\perp+A_1^\perp(A_1/t)+\frac\d{\d t}A_1^\perp
	\right]\Phi_u(t,0)=0.
\]
There is no sign in front of $A_1^\perp(A_1/t)$ since $\d/\d t$ is an even operator.
Combine this with the following lemma.
\begin{lem}\label{second lemma}  
$(A_pA_q)/t=A_p^\perp(A_q/t)-(A_p/t)A_q^\perp$.
\end{lem}
We get:
\[
	\int_0^1dt\, \Phi_u(1,t)\left[
	(A_1A_1)/t+\frac\d{\d t}A_1^\perp
	\right]\Phi_u(t,0)=0.
\]
Combine this with the integral just after Lemma \ref{first lemma} to get that
\[
	\int_0^1dt\,\Phi_u(1,t)\left
	[(A_0A_2+A_2A_0)/t-(dA_1)/t+(A_1A_1)/t
	\right]\Phi_u(t,0)=0
\]
is the condition equivalent to the boxed equation above.

\begin{thm}\label{previous thm} Suppose that $A_0^2=0$ and $dA_0=A_0A_1+A_1A_0$ at all points of $M$. Then, the following are equivalent.
\begin{enumerate}
\item The integral of $\Psi_1(1,0)$ gives a cochain homotopy $\Phi_0(1,0)\simeq \Phi_1(1,0)$ for all smooth homotopies of smooth paths from $x_0$ to $x_1$.
\item For any $\g\in\Path(M,x_0,x_1)$ we have
\[
	A_0\Psi_1-\Psi_1A_0=d\Phi
\]
as elements of $\Omega^1(\Path(M,x_0,x_1),\Hom^0(V_{x_0},V_{x_1}))$.
\item 
\[
	\boxed{
	dA_1=A_0A_2+A_1A_1+A_2A_0
	}
\]
at all points in $M$ in the connected component of $x_0,x_1$. 
\end{enumerate}
\end{thm}

\begin{proof} (1) and (2) are equivalent since (2) is the infinitesimal version of (1). The above calculations show that (3) implies (2). To see that Condition (3) is necessary, suppose there is a point $z\in M$ in the component of $x_0$ so that this conditions fails. Since this is an equation of 2-forms, there must be some tangent 2-plane at $z$ on which the equation fails. Choose a path that goes from $x_0$ to $z$ and then to $x_1$. Perform a homotopy which slides along this 2-plane in a small neighborhood of $z$ and is otherwise stationary. Then the integral above will not be zero and we get a contradiction, i.e., we do not get a cochain homotopy as required. So, the boxed equation (3) and the previous boxed equation (2) are equivalent when universally quantified.
\end{proof}


\subsection{Higher homotopies}

Suppose now that we have a smooth map
\[
	h:(I^{k+1},0\times I^k,1\times I^k)\to (M,x_0,x_1)
\]
which we view as a $k$-parameter family of paths $\g_u(t), t\in I, u\in I^k$, from $x_0$ to $x_1$. Then, for $p\le k$, the superconnection parallel transport gives a matrix $p$-form
\[
	\Psi_p=\Psi_p(1,0)\in \Omega^p(I^k,\Hom^{-p}(W_0,W_1)).
\]
The integral of this $p$-form on any $p$-dimensional face of $I^k$ gives a degree $-p$ homomorphism
\[
	W_0=V_{x_0}\to W_1=V_{x_1}.
\]
What we want is for the integral of $\Psi_k$ over $I^k$ to be the null-homotopy of the alternating sum of the integral of $\Psi_{k-1}$ over the codimension one faces of $I^k$, i.e.,
\[
	A_0\int_{I^k}\Psi_k(1,0)-(-1)^k\left(\int_{I^k}\Psi_k(1,0)\right)A_0=\int_{\d I^k}\Psi_{k-1}(1,0).
\]
The LHS is the integral of $A_0\Psi_k-\Psi_kA_0$. So, by Stokes' Theorem this condition for all $h$ is equivalent to the following.
\[
	\boxed{
	A_0\Psi_k-\Psi_kA_0=d\Psi_{k-1}
	}
\]
If we expand $\Psi_k=\Psi_k(1,0)$ as an iterated integral as in \ref{cor: second iterated integral for Psi}, the LHS of this boxed equation becomes:
\[
	\LHS=\int_0^1dt\sum_{
	p+q+r=k,\ q\ge1
	}\Psi_p(1,t)\left[
	A_0(A_{q+1}/t)-(A_{q+1}/t)A_0
	\right]
	\Psi_r(t,0)
\]
and the RHS of the boxed equation becomes:
\[
	d\Psi_{k-1}=\int_0^1dt\sum_{
	p+q+r=k,\ q\ge1
	}\Psi_p(1,t)\left[
	d(A_{q}/t)
		\right]
	\Psi_r(t,0).
\]
By Lemma \ref{first lemma} this is equal to
\[
	\RHS=\int_0^1dt\sum_{
	p+q+r=k,\ q\ge1
	}\Psi_p(1,t)\left[
	\frac\d{\d t}A_q^\perp+(dA_q)/t
		\right]
	\Psi_r(t,0).
\]

As before we consider the $k$-parameter family of functions
\[
	f_u(t)=\sum_{p+q+r=k,\ q\ge1}\Psi_p(1,t)A_q^\perp(t,u) \Psi_r(t,0), u\in I^k.
\]
Since $A_q^\perp=0$ at $t=0,1$ we have
\[
	\int_0^1dt\,\frac\d{\d t}f_u(t)=0
\]
for all $u$. The derivative of $f_u(t)$ has three terms:
\[
	\frac\d{\d t}f_u(t)=\sum -\Psi_i(A_j/t)A_q^\perp \Psi_r
	+\Psi_p\left(\frac\d{\d t}A_q^\perp\right)\Psi_r
	+\Psi_p A_q^\perp \left(A_i/t\right)\Psi_j.
\]
If we change the names of the indices we can collect terms to get:
\[
	\frac\d{\d t}f_u(t)=\sum \Psi_p\left[
	\sum_{i+j=q+1} -(A_i/t)A_j^\perp+
	\frac\d{\d t}A_q^\perp
	+\sum_{i+j=q+1}A_i^\perp (A_j/t)
	\right]
	\Psi_r
\]
\[
	=\sum_{p+q+r=k,\ q\ge1} \Psi_p\left[
	\sum_{i+j=q+1} (A_iA_j)/t+
	\frac\d{\d t}A_q^\perp
	\right]
	\Psi_r
\]where $i,j\ge1$ in the sum.
Therefore, the RHS of the boxed equation is
\[
	\int_0^1dt\sum_{
	p+q+r=k,\ q\ge1
	}\Psi_p\left[
	\sum_{i=1}^q -(A_iA_{q-i+1})/t+(dA_q)/t
		\right]
	\Psi_r.
\]
If we compare this with the LHS we get the following theorem whose proof is the same as the proof of the previous theorem \ref{previous thm}.

\begin{thm}\label{thm:flat superconnections give higher homotopies} Suppose that $A_0^2=0$ at all points of $M$ and $k\ge1$. Then, the following three conditions are equivalent.
\begin{enumerate}
\item The integral of $\Psi_q$ over the cube $I^q$ gives the null homotopy of the alternating sum of the integrals of $\Psi_{q-1}$ over the faces of $I^q$, i.e.,
\[
	A_0\int_{I^q}\Psi_q-(-1)^q\left(\int_{I^q}\Psi_q\right)A_0=\int_{\d I^q}\Psi_{q-1}
\] for all maps $h:I^q\to\PathMx$ for all $0\le q\le k$.
\item The following holds at all paths $\g\in\Path(M,x_0,x_1)$ and for all $0\le q\le k$:
\[
	{
	A_0\Psi_q-\Psi_qA_0=d\Psi_{q-1}
	}
\]
\item The following condition holds for all $0\le q\le k$ at all points of $M$ in the component of $x_0,x_1$:
\[
	\boxed{
	dA_q=\sum_{i=0}^{q+1}A_iA_{q-i+1}
	}
\]
\end{enumerate}
Furthermore, these equivalent conditions, plus the condition $A_0^2=0$, hold for all $k$ if and only if $D$ is flat, i.e., $D^2=0$.
\end{thm}

\begin{defn}
We say that the superconnection $D=\nabla-A_0-A_2-A_3-\cdots$ on $V$ is \emph{flat} if $D^2=0$. This is equivalent to the boxed equation above for all $q\ge-1$ at all points of $M$ (where $A_{-1}=0$) since
\[
	d\circ A+A\circ d=dA.
\]
Note that the condition $D^2=0$ implies all of the boxed equations in this section.
\end{defn}


\subsection{Hidden signs} There are hidden signs in the matrix forms $A_pA_q$ and in the superconnection parallel transport $\Psi_k$. They appear when we collect the $\End(V)$ and $\Omega^\ast(M)$ terms. For example, the product $A_pA_q$ has a hidden sign of $(-1)^{p(q-1)}$:
\[
	A_pA_q=\sum (f_i\otimes\a_i)(g_j\otimes\b_j)=\sum(-1)^{p(q-1)}f_ig_j\otimes\a_i\b_j
\]
We are also using supercommutation rules for differentiation:
\[
	dA_p=d\sum_i f_i\otimes\a_i=(-1)^{p-1} \sum_{i,j} \frac{\d}{\d x_j}f_i\otimes dx_j \a_i+(-1)^{p-1}\sum_i f_i\otimes d\a_i
\]
The signs in front of the two sums on the right must be equal since the tensor product is over $\Omega(M)$ and the sign on the second sum is given by the Koszul sign rule.
The superconnection parallel transport has a more complicated hidden sign:
\[
	\Psi=\sum\iint_{t\ge u_1\ge\cdots\ge  u_n\ge s}du_1\cdots du_n\,
	(A_{k_1}/u_1)(A_{k_2}/u_2)\cdots(A_{k_n}/u_n)
\]
\[
	=\sum\iint du_1\cdots du_n\,(f_{i_1}\otimes\a_{i_1})\cdots(f_{i_n}\otimes\a_{i_n})
\]
\[
	=\sum (-1)^x\iint du_1\cdots du_n\,f_{i_1}\cdots f_{i_n}\otimes\a_{i_1}\cdots\a_{i_n}
\]
where
\[
	x=\sum_{1\le i<j\le n}(k_i-1)(k_j-1).
\]
And, as we saw before, integration does not commute with multiplication:
\[
	\left(\int_{I^k}\Psi_k\right)A_0=(-1)^k\int_{I^k}\left(\Psi_k A_0\right).
\]
This will become important in the next section.

%
%

\section{$A_\infty$ functors and twisting cochains}\label{sec 4: A-infty and twisting}

If we integrate a flat superconnection on a graded vector bundle $V$ over smooth simplices in $M$ we get an $A_\infty$ functor. In local coordinates this becomes a twisting cochain. With the properties of superconnection parallel transport that we have derived so far, this will all be fairly easy.


\subsection{Twisting cochains} We will go over the general definition of a twisting cochain (in the cohomology setting) and then go to the particular case of twisting cochains on $C=C_\ast(M;K)$ with coefficients in $E=\End(V_\ast)$. First, the basic definitions.


\subsubsection{cup product}

We define the cup product of two cochains $\f,\psi$ on a differential graded coalgebra $C$ with coefficients in a differential graded algebra $E$ by
\[
	\f\cup' \psi:=\mu(\f\otimes\psi)\Delta
\]
where $\Delta:C\to C\otimes C$ is the diagonal map and $\mu:E\otimes E\to E$ is the product map. This is not the standard definition since signs appear when the terms are expanded. However, it has the advantage of simplicity. For example, it is obviously associative if $\Delta$ is coassociative and $\mu$ is associative. The standard cup product (without signs in its expanded form) will be denoted $\cup$ (without the prime) as in the third paragraph below.

The particular case we are considering is the coalgebra $C=C_\ast(M;K)$ of (smooth) singular chains in the manifold $M$ with coefficients in a commutative ring $K$. By this we mean that $C_n$ is the free $K$-module generated by all smooth mappings
\[
	\s:\Delta^n\to M
\]
where $\Delta^n$ is the linear simplex in Euclidean space spanned by a fixed set of vertices $v_0,v_1,\cdots,v_n$ which will be specified later. 

The diagonal map on $C$ is given in the standard way by
\[
	\Delta \s=\sum_{p+q=n} f_p\s\otimes b_{q}\s
\]
where $f_p\s=\s[v_0,\cdots,v_p]$ is the front $p$-face of $\s$ and $b_q\s=\s[v_p,\cdots,v_n]$ is the back $q$-face of $\s$. We use the notation that
\[
	[v_{j_0},\cdots,v_{j_k}]:\Delta^k\to\Delta^n
\]
for the affine linear map which sends the vertex $v_i\in\Delta^k$ to $v_{j_i}\in\Delta^n$.

If $\f\in C^p(M;K)=\Hom_K(C_p(M;K),K)$ and $\psi\in C^q(M;K)$ then
\[
	(\f\cup'\psi)\s=\mu(\f\otimes \psi)\Delta \s
\]
\[
	=\mu(\f\otimes \psi)(f_p\s\otimes b_q\s)=(-1)^{pq}\f(f_p\s)\psi(b_q\s)
\]
\[
	=(-1)^{pq}(\f\cup\psi)\s.
\]
The sign appears when the $q$-cochain $\psi$ and the $p$-chain $f_p\s$ go past each other.


\subsubsection{twisting cochains in general}

\begin{defn} Let $C$ be a differential nonnegatively graded coalgebra over a commutative ring $K$ and let $E$ be a differential $\ZZ$-graded algebra over $K$ which we index with two different notations: $E_n=E^{-n}$. Then a \emph{twisting cochain} on $C$ with coefficients in $E$ is a $K$-linear map $\t:C\to E_\ast$ of degree $-1$, i.e., $\t=\sum_{n\ge0}\t_n$ where  $\t_n:C_n\to E_{n-1}=E^{1-n}$ so that $\t_0=0$ and
\[
\d \t=\t\d-\t\cup'\t.
\]
\end{defn}

Now consider the special case when the differential $\d:E^n\to E^{n+1}$ in $E$ is given by commutation with a square zero element $d\in E^1$:
\[
	\d a=da-(-1)^n ad
\]
for all $a\in E_n$. If $c\in C_n, n>0$ then $\t c\in E_{n-1}$. So
\[
	\d(\t_n c)=d(\t_nc)+(-1)^n(\t_nc)d.
\]
The definition of twisting cochain can now be written:
\[
	d(\t_nc)+(-1)^n(\t_nc)d=\t_{n-1}(\d c)-\sum(-1)^{\deg c_{(1)}}\t(c_{(1)})\t(c_{(2)})
\]
where we used the Sweedler notation \cite{Sweedler}:
\[
	\Delta c=\sum_{(c)}c_{(1)}\otimes c_{(2)}
\]
with the index of summation $(c)$ suppressed.
In the special case $n=1$, the RHS is zero and the equation becomes:
\[
	d(\t_1c)=(\t_1c)d.
\]
This is assuming that $d\in E^1$ is fixed. However, in our case $d=A_0$ is variable and the situation is a little different.


\subsubsection{variable $d$}

We need $d$ to be a function of $c$. Loosely speaking, it should be a function of the ``first vertex'' of $c$ when it is on the left and the ``last vertex'' of $c$ when it is on the right. More precisely, $d$ will be a $K$-linear mapping:
\[
	d:C_0\to E^1=E_{-1}
\]
so that $\ov\t=d+\t$ is a $K$-linear map $C\to E_\ast$ of degree $-1$ satisfying the following condition.
\[
	\boxed{\ov\t\d=\ov\t\cup'\ov\t}
\]
The equation $\ov\t=d+\t$ is the same as saying $\ov\t_0=d$ and $\ov\t_k=\t_k$ for all $k>0$. We will call $\ov\t$ an \emph{augmented twisting cochain} if it satisfies the above boxed equation. We will now examine this new boxed formula to determine:
\begin{enumerate}
\item that it reduces to the previous definition in the case when $d$ is fixed and
\item to see what it says when $d$ is variable in the case $C=C_\ast(M;K)$.
\end{enumerate}

\begin{prop}
Suppose that $d\in E^1$ is fixed and $\ov\t_0:C_0\to E^1$ is given by
\[
	\ov\t_0(c)=\e(c)d
\]
where $\e:C_0\to K$ is the counit of $C$. Then $\ov\t=\ov\t_0+\t$ is an augmented twisting cochain if and only if $d^2=0$ and $\t$ is a twisting cochain where the differential on $E$ is given by commutation with $d$.
\end{prop}

\begin{proof}
If we apply the operators in the boxed equation to $c\in C$ we get:
\[
	\ov\t\d c=\sum(-1)^{\deg c_{(1)}} \ov\t(c_{(1)})\ov\t(c_{(2)}).
\]
When $c\in C_0$ the LHS is 0 and the RHS is $\e(c)d^2$. So, the equation holds for all $c\in C_0$ if and only if $d^2=0$. If $c\in C_n$ for $n>0$ we use the definition of a counit which says:
\[
	c=\sum\e(c_{(1)})c_{(2)}=\sum c_{(1)}\e(c_{(2)}).
\]
The first equation (which says $\e$ is a left counit) gives
\[
	d\t(c)=\sum \underbrace{\e(c_{(1)})d}_{\ov\t_0(c_{(1)})}\t(c_{(2)})=(\ov\t_0\cup'\t)(c).
\]
The second (which says $\e$ is a right counit) gives
\[
	\t(c)d=\sum \t(c_{(1)})\underbrace{\e(c_{(2)})d}_ {\ov\t_0(c_{(2)})} =(-1)^{n}(\t\cup'\ov\t_0)(c).
\]
Therefore,
\[
	(\ov\t\cup'\ov\t)(c)=d\t(c)+(-1)^n\t(c)d+(\t\cup'\t)(c)
\]
which transforms the boxed equation into the equation in the definition of a twisting cochain.
\end{proof}


\subsubsection{twisting cochains on $M$}
We are particularly interested in the case where $C=C_\ast(M;K)$ and $E=\End(V)$. The reason is that we are trying to understand the relationship between the higher Reidemeister torsion of \cite{IBookOne}, \cite{IComplexTorsion} and the higher analytic torsion of \cite{BismutLott95}, \cite{BG2}, \cite{Goette01}, \cite{Goette03}. The higher Reidemeister torsion is a sequence of cohomology classes on $M$ which comes from integrating flat superconnections corresponding to smooth families of acyclic chain complexes over $M$. See \cite{ITwisted} for further details on this interpretation of the higher torsion. When explicit formulas for these superconnections are given, the integral can sometimes be computed. For example, see \cite{IBookOne}, Chapter 7.

In the case where $C=C_\ast(M;K)$, $C_0=C_0(M;K)$ is freely generated by 0-simplicies $v:\Delta^0\to M$ which we identify with their images in $M$ and call ``vertices.'' We have the following proposition which follows directly from the definitions.

\begin{prop}\label{prop:expand augmented twisting cochain}
Suppose that $\ov\t:C_\ast(M;K)\to E$ is an augmented twisting cochain. Then for any vertex $v\in C_0(M;K)$, $d_v=\ov\t_0(v)$ is a square-zero element of $E^1$. Furthermore, if $d_v\in E^1, v\in M$, is any family of square-zero elements  and $\t=\sum_{i\ge1}\t_i,\t_i:C_i\to E_{i-1}$, then $\ov\t=d+\t$ is an augmented twisting cochain (i.e., satisfies the boxed equation above) if and only if the following holds for all simplices $\s:\Delta^n\to M$, $n\ge 1$.
\[
	d_{\s(v_0)}(\t\s)+(-1)^n(\t\s)d_{\s(v_n)}=\ov\t(\d \s)-\sum_{p=1}^{n-1}(-1)^p\t(\s[v_0,\cdots,v_p])\t(\s[v_p,\cdots,v_n]).
\]
\end{prop}

Let us look at the case $n=1$. Then $\s:\Delta^1\to M$ is a path connecting $x_0=\s(v_0)$ to $x_1=\s(v_1)$ and the above equation becomes:
\[
	d_{x_0}(\t_1\s)-(\t_1\s)d_{x_1}=\ov\t_0\d\s=d_{x_1}-d_{x_0}.
\]
We can rewrite this as:
\[
	\boxed{d_{x_0}(1+\t_1\s)=(1+\t_1\s)d_{x_1}}
\]
which can be written
\[
	d_{x_0}\f_\s=\f_\s d_{x_1}
\]
when we use the abbreviation $\f(\s)=\f_\s:=1+\t_1\s$. Note that the linear extension of $\f$ to $C_0(M;K)$ is
\[
	\f=u\circ\e+\t_1:C_0(M;K)\to E_0.
\]
where $\e:C_0\to K$ is the counit (the augmentation map) and $u:K\to E_0$ is the unit (multiplication by $1\in E_0$).

When $n\ge3$ something similar happens. Namely, there will always be two pairs of terms in the equation in Proposition \ref{prop:expand augmented twisting cochain} which can be combined. It is helpful to count the number of terms in the equation. There are two terms on the left and there are $2n$ terms on the right since $\d\s$ has $n+1$ terms. When we combine two pairs, this will be reduced to $2n-2$ terms on the right. When $n=2$ we also get a reduction of $2n=4$ terms to $2n-2=2$ terms but in a slightly different way.

Here is the calculation for $n\ge3$.
\[
	\ov\t\d\s=\t\d\s=\sum_{i=0}^n (-1)^i \t(\s[v_0,\cdots,\what{v_i},\cdots,v_n])
\]
\[
	=\t(\s[v_1,\cdots,v_n])+(-1)^n\t(\s[v_0,\cdots,v_{n-1}])+\sum_{i=1}^{n-1}(-1)^i\t(\d_i\s).
\]
We can combine the $i=0$ and $i=n$ terms with the $p=1$ and $p=n-1$ terms in the equation which are:
\[
	\t_1(\s[v_0,v_1])\t(\s[v_1,\cdots,v_n])+(-1)^n\t(\s[v_0,\cdots,v_{n-1}])\t_1(\s[v_{n-1},v_n]).
\]
The equation in Proposition \ref{prop:expand augmented twisting cochain} can now be rewritten with 2 fewer terms as follows.
\[
	d(\t\s)+(-1)^n(\t\s)d=\f(f_1\s)\t(b_{n-1}\s)+(-1)^n\t(f_{n-1}\s)\f(b_1\s)
\]
\[
	+\sum_{i=1}^{n-1}(-1)^i\t(\d_i\s)-\sum_{p=2}^{n-2}\t(f_p\s)\t(b_{n-p}\s).
\]

In the case $n=2$ we get:
\[
	d(\t\s)+(\t\s)d=\f(f_1\s)\f(b_1\s)-\f(\d_1\s).
\]
We summarize this calculation as follows.

\begin{cor}\label{characterization of augmented twisting cochains}
Let $\psi_i:C_i(M;K)\to E^{1-i}, i\ge0$ be any sequence of $K$-linear map. Let $\psi=\sum \psi_i$. Then $\ov\t=\psi-u\circ \e$ is an augmented twisting cochain on $C_\ast(M;K)$ with coefficients in $E_\ast$ if and only if
\[
	\sum_{i=1}^{n-1}(-1)^i\psi_{n-1}(\d_i c)=(\psi\cup'\psi)c
\]
for all $c\in C_n(M;K),n\ge0$.
\end{cor}


\subsubsection{coefficients in $E=\End(V_\ast)$}

Suppose that $V_\ast$ is a fixed $\ZZ$-graded $K$-module which we also write as $V^n=V_{-n}$. Then $E=\End(V)$ is the $\ZZ$-graded $K$-algebra $E=\bigoplus E_n$ where
\[
	E_n=E^{-n}=\prod \Hom(V_i,V_{i+n}).
\]
If $d_v\in E^1$, $v\in M$, is a family of square-zero endomorphisms of $V$ then
\[
	V_v=(V^\ast,d_v)
\]
is a family of cochain complexes parametrized by $v\in M$. The last boxed equation says that, for all smooth 1-simplices $\s:\Delta^1\to M$, 
\[
	\f_\s=1+\t_1\s:V_{\s(v_1)}\to V_{\s(v_0)}
\]
is a cochain map.

The $n=2$ equation above says that, for 2-simplices $\s:\Delta^2\to M$, $\t_2\s$ is a cochain homotopy:
\[
	\t_2\s:\f_{\s[v_0,v_2]}\simeq \f_{\s[v_0,v_1]}\f_{\s[v_1,v_2]}.
\]

For $n\ge3$, $\t_n\s$ is a ``higher homotopy.'' We will come back to this shortly.


\subsection{Families of paths in $\Delta^k$}\label{subsec: integration over simplices}

We want to integrate flat superconnections on simplices. We know how to integrate them over cubes. To integrate over a simplex we just need a map
\[
	\th:I^k\to\Delta^k
\]
which we view as a family of paths $\th_w:I\to \Delta^k$ parametrized by $w\in I^{k-1}$:
\[
	\theta_{(k)}:I^{k-1}\to \Path\Delta^k.
\]
This mapping is not the same as the mapping given \cite{KadSane05} but it has similar properties. It is a composition of two piecewise linear maps:
\[
	I^k\xrarrow{\ll} I^k\xrarrow{\pi_k}\Delta^k.
\]

First we describe $\pi_k$.
\[
	\Delta^k:=\{y\in \RR^k\st 1\ge y_1\ge y_2 \ge\cdots\ge y_k\ge0\}.
\]
This is a closed subset of $I^k$ and a retraction $\pi_k:I^k\to\Delta^k$ is given by $\pi_k(x)=y$ where
\[
	y_i=\max(x_i,x_{i+1},\cdots,x_k).
\]
We note that the vertices of $\Delta^k$ are $v_0=0$ and
\[
	v_i=(\overbrace{1,1,\cdots,1}^i,0,\cdots,0)=\pi_k(e_i)
\]
and barycentric coordinates on $\Delta^k$ are given by $t_i=y_i-y_{i+1}$ where $y_0=1$ and $y_{k+1}=0$. 

Here are some trivial observations about the mapping $\pi_k$.
\begin{prop}\begin{enumerate}
\item $\pi_k$ is order preserving, i.e., if $x\le x'$ in the sense that $x_i\le x_i'$ for all $i$ then $\pi_k(x)\le\pi_k(x')$. Furthermore, $\pi_k(x)\ge x$.
\item $\pi_k$ sends $\d_i^+I^k=\{x\in I^k\st x_i=1\}=\{x\ge e_i\}$ onto the back $k-i$ face of $\Delta^k$ (the face spanned by $v_i,v_{i+1},\cdots,v_k$ and given by the equation $y\ge v_i$).
\item $\pi_k$ sends $\d_i^-I^k=\{x\in I^k\st x_i=0\}$ onto $\d_i\Delta^k=\{y\in\Delta^k\st y_i=y_{i+1}\}$.
\end{enumerate}
\end{prop}

Next we construct the mapping $\ll:I^k\to I^k$ as a family of paths $\ll_w:I\to I^k$ parametrized by $w\in I^{k-1}$:
\[
	\ll_{(k-1)}:I^{k-1}\to\Path I^k.
\]
If $w=(w_1,w_2,\cdots,w_{k-1})$ then $\ll_w$ will be the path which goes backwards through the $k+1$ points:
\[
	0\ot w_1e_1\ot w_1e_1+w_2e_2\ot\cdots\ot \sum_{i=1}^kw_ie_i
\]
where $w_k=1$. In other words,
\[
	\ll_w\left(1-\frac jk\right)=\sum_{i=1}^jw_ie_i
\]
for $j=0,1,2,\cdots,k$ and we interpolate linearly by
\[
		\ll_w\left(1-\frac {j+s}k\right)=\sum_{i=1}^jw_ie_i+sw_{j+1}e_{j+1}
\]
if $0<s<1$.

We define $\theta_{(k)}$ to be the composition
\[
	\theta_{(k)}=P\pi_k\circ\ll_{(k-1)}:I^{k-1}\to \Path \Delta^k
\]
Here are some elementary properties of this second mapping.

\begin{prop}
\begin{enumerate}
\item The adjoint of $\theta_{(k)}$ is a piecewise linear epimorphism $I^k\onto \Delta^k$.
\item For each $w\in I^{k-1}$, $\th_w$ is a path from $\th_w(0)=v_k$ to $\th_w(1)=v_0$.
\item $\th_w$ passes through the vertex $v_i$ if and only if $w_i=1$.
\item $\theta_{(k)}$ takes each of the $2^{k-1}$ vertices of $I^{k-1}$ to the shortest path from $v_k$ to $v_0$ passing through the corresponding subset of $\{v_1,\cdots,v_{k-1}\}$.
\end{enumerate}
\end{prop}

We will draw the paths $\theta_{(k)}$ for $k=2,3$.

In the case $k=2$ and $0<w<1$ the path $\ll_w$ in $I^2$ is two line segments going through the points:
\[
	\ll_w: (w,1)\to (w,0)\to (0,0)
\]
and $\pi_2$ retracts this path back into the 2-simplex giving three line segments:


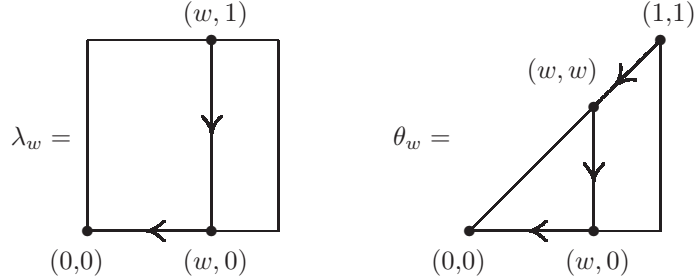
\def\arrowheadup{\qbezier(0,.1)(0,.05)(.05,0)
      \qbezier(0,.1)(0,.05)(-.05,0)}
\def\arrowheaddown{\qbezier(0,0)(0,.05)(.05,0.1)
      \qbezier(0,0)(0,.05)(-.05,0.1)}
\def\arrowheadleft{\qbezier(0,0)(.05,0)(.1,.05)
      \qbezier(0,0)(.05,0)(.1,-.05)}
\def\arrowheadright{\qbezier(0.1,0)(.05,0)(0,.05)
      \qbezier(0.1,0)(.05,0)(0,-.05)}
\def\arrowheaddownleft{\qbezier(0,0)(0.05,.05)(.04,0.12)
      \qbezier(0,0)(0.05,.05)(.12,0.04)}

\begin{figure}[htbp]
\begin{center}
%
\setlength{\unitlength}{1in}
\mbox{
\begin{picture}(4,1.5)
	\thinlines
	\put(0.3,0){(0,0)}
	\put(0.5,.2){\line(1,0){1}}
	\put(1.5,.2){\line(0,1){1}}
	\put(0.5,.2){\line(0,1){1}}
	\put(0.5,1.2){\line(1,0){1}}
   \thicklines
      \put(1,1.3){$(w,1)$}
      \put(1.15,1.2){\circle*{.05}} 
      \put(1.15,.2){\circle*{.05}} 
      \put(1.15,.2){\line(0,1){1}} 
      \put(1.15,.7){\arrowheaddown}
      \put(.8,.2){\arrowheadleft}
      \put(.5,.2){\line(1,0){.65}} 
      \put(.5,.2){\circle*{.05}} 
      \put(1,0){$(w,0)$}
      \put(0.1,.65){$\ll_w=$}
      \thinlines
      \put(2,0){      
      \put(0.1,.65){$\th_w=$}
	\put(0.3,0){(0,0)}
	\put(0.5,.2){\line(1,0){1}} 
	\put(1.5,.2){\line(0,1){1}}  
	\qbezier(0.5,.2)(1,.7)(1.5,1.2)  
	     \put(1,0){$(w,0)$}
	     \thicklines
	          \put(1.15,.85){\circle*{.05}}  
	               \put(1.15,.2){\line(0,1){.65}} 
	          \put(1.15,.2){\circle*{.05}}  
	               \put(1.15,.45){\arrowheaddown}
	           \put(.8,.2){\arrowheadleft}
      \put(.5,.2){\line(1,0){.65}}  
      \qbezier(1.15,.85)(1.3,1)(1.5,1.2) 
      \put(.5,.2){\circle*{.05}}  
      \put(1.5,1.2){\circle*{.05}}  
           \put(.8,1){$(w,w)$}
           \put(1.4,1.3){(1,1)}
           \put(1.25,.95){\arrowheaddownleft}
      }
\end{picture}
}
\end{center}
\end{figure}
%


Thus we see that $\theta_{(2)}$ is a homotopy:

%
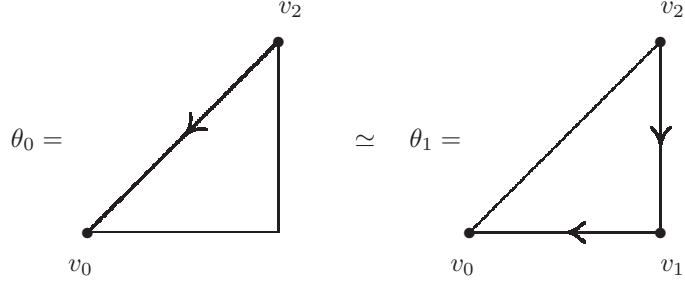
\begin{figure}[htbp]
\begin{center}
%
\setlength{\unitlength}{1in}
\mbox{
\begin{picture}(4,1.5)
      \thinlines
      {      
      \put(0.1,.65){$\th_0=$}
	\put(0.4,0){$v_0$}
	\put(1.5,1.35){$v_2$}
	\put(0.5,.2){\line(1,0){1}} 
	\put(1.5,.2){\line(0,1){1}}  
	\qbezier(0.5,.2)(1,.7)(1.5,1.2)  
	     \thicklines
      \qbezier(.5,.2)(1.3,1)(1.5,1.2) 
      \put(.5,.2){\circle*{.05}}  
      \put(1.5,1.2){\circle*{.05}}  
           \put(1,.7){\arrowheaddownleft}
      }
      \thinlines
      \put(2,0){      
      \put(-.1,.65){$\simeq \quad \th_1=$}
	\put(0.4,0){$v_0$}
	\put(1.5,0){$v_1$}
	\put(1.5,1.35){$v_2$}
	\qbezier(0.5,.2)(1,.7)(1.5,1.2)  
	     \thicklines
	     	\put(1.5,.2){\line(0,1){1}}  
	          \put(1.5,.2){\circle*{.05}}  
	               \put(1.5,.65){\arrowheaddown}
	           \put(1,.2){\arrowheadleft}
      \put(.5,.2){\line(1,0){1}}  
      \put(.5,.2){\circle*{.05}}  
      \put(1.5,1.2){\circle*{.05}}  
      }
\end{picture}
}
\end{center}
\end{figure}
%

In the case $k=3$, the path $\ll_w$ in $I^3$ is given by three line segments going through the points
\[
	\ll_w:(w_1,w_2,1)\to (w_1,w_2,0)\to (w_1,0,0)\to (0,0,0).
\]
The mapping
\[
	\theta_{(3)}:I^2\to\Path\Delta^3
\]
is a higher homotopy in the sense that the restriction of $\theta_{(3)}$ to each of the four sides of the square $I^2$ is the homotopy $\theta_{(2)}$ on one of the four sides of the tetrahedron $\Delta^3$. Well, this is not exactly true. On $\d_1^+I^2$ (where $w_1=1$), the homotopy $\theta_{(2)}$ on the face $\d_0\Delta^3$ does not contain the vertex $v_0$. So, we need to add a line segment from $v_1$ to $v_0$. This is shown in black in the figure below. Similarly, on $\d_2^+I^2$, the homotopy $\theta_{(2)}$ on the face $\d_3\Delta^3$ does not contain $v_3$. So, we need to add a line segment from $v_3$ to $v_2$ as shown in red in the figure.



\def\tetrahedron{{\thinlines\color{green}  
  \put(0,0){\line(1,0){.6}}  
  \qbezier(.6,0)(.7,.1)(.9,.3) 
  \put(.9,.3){\line(0,1){.6}} 
  \qbezier(0,0)(.3,.3)(.9,.9) 
  \qbezier(0,0)(.3,.1)(.9,.3)  
  \qbezier(.6,0)(.7,.3)(.9,.9)  
    }}
\def\tetrahedronB{{\thinlines\color{green}  
  \put(0,0){\line(1,0){.6}}  
  \qbezier(.6,0)(.7,.1)(.9,.3) 
  \qbezier(0,0)(.3,.3)(.9,.9) 
  \qbezier(0,0)(.3,.1)(.9,.3)  
  \qbezier(.6,0)(.7,.3)(.9,.9)  
    }}
\begin{figure}[htbp]
\begin{center}
%
{
\setlength{\unitlength}{2.5cm}
{\mbox{
\begin{picture}(5.5,5.2)
	\put(0,.3){
		\thinlines
	\put(.2,-.4){$w_1=0$}
	\put(1.5,-.4){$w_1=\frac13$}
	\put(2.8,-.4){$w_1=\frac23$}
	\put(4.1,-.4){$w_1=1$}
	\put(5.2,.4){$w_2=0$}
	\put(5.2,1.7){$w_2=\frac13$}
	\put(5.2,3){$w_2=\frac23$}
	\put(5.2,4.3){$w_2=1$}
	%
    \put(0,0){\tetrahedron \thicklines\qbezier(0,0)(.3,.3)(.9,.9)}
    \put(0,1.3){\tetrahedron \thicklines\qbezier(.3,.3)(.7,.7)(.9,.9) 
     \qbezier(0,0)(.15,.05)(.3,.1) }
    \put(0,2.6){\tetrahedron \thicklines
     \qbezier(0,0)(.3,.1)(.6,.2)}
    \put(0,3.9){\tetrahedronB \thicklines
     \qbezier(0,0)(.3,.1)(.9,.3)}
    \put(1.3,0){\tetrahedron  \thicklines\put(0,0){\line(1,0){.2}} 
    {\color{blue}
     \qbezier(.2,0)(.25,.15)(.3,.3) } 
      \qbezier(.3,.3)(.7,.7)(.9,.9)  }
          \put(0,1.3){\thicklines {\color{red}\put(.3,.1){\line(0,1){.2}} }}
    \put(1.3,1.3){\tetrahedron  \thicklines\put(0,0){\line(1,0){.2}}
    {\color{blue}
     \qbezier(.2,0)(.25,.05)(.3,.1) }}
    \put(1.3,2.6){\tetrahedron  \thicklines\put(0,0){\line(1,0){.2}}
     \qbezier(.3,.1)(.45,.15)(.6,.2){\color{blue}
     \qbezier(.2,0)(.25,.05)(.3,.1) }}
    \put(1.3,3.9){\tetrahedronB  \thicklines\put(0,0){\line(1,0){.2}}
     \qbezier(.3,.1)(.6,.2)(.9,.3)
     {\color{blue}
     \qbezier(.2,0)(.25,.05)(.3,.1) }}
     \put(1.3,1.3){\thicklines {\color{red}\put(.3,.1){\line(0,1){.2}}}
      \qbezier(.3,.3)(.7,.7)(.9,.9) }
    \put(2.6,0){\tetrahedron \thicklines\put(0,0){\line(1,0){.4}} 
    {\color{blue}
     \qbezier(.4,0)(.5,.3)(.6,.6) }  
      \qbezier(.6,.6)(.7,.7)(.9,.9) }
    \put(2.6,1.3){\tetrahedron \thicklines\put(0,0){\line(1,0){.4}}
    {\color{blue}
     \qbezier(.4,0)(.45,.05)(.5,.1) \qbezier(.5,.3)(.55,.45)(.6,.6)}}
    \put(2.6,1.3){\thicklines {\color{red}\put(.5,.1){\line(0,1){.2}}
    } 
     \qbezier(.6,.6)(.7,.7)(.9,.9)}
    \put(2.6,2.6){\tetrahedron \thicklines\put(0,0){\line(1,0){.4}}
    {\color{blue}
     \qbezier(.4,0)(.5,.1)(.6,.2) }  }
    \put(0,2.6){\thicklines {\color{red}\put(.6,.2){\line(0,1){.4}}
    }
    \qbezier(.6,.6)(.7,.7)(.9,.9) }
    \put(1.3,2.6){\thicklines {\color{red}\put(.6,.2){\line(0,1){.4}}} 
     \qbezier(.6,.6)(.7,.7)(.9,.9)}
    \put(2.6,2.6){\thicklines {\color{red}\put(.6,.2){\line(0,1){.4}}}
     \qbezier(.6,.6)(.7,.7)(.9,.9) }
    \put(0,3.9){\thicklines {\color{red}\put(.9,.3){\line(0,1){.6}}
    } }
   \put(1.3,3.9){\thicklines {\color{red}\put(.9,.3){\line(0,1){.6}}} }
   \put(2.6,3.9){\thicklines {\color{red}\put(.9,.3){\line(0,1){.6}}} 
   }
   \put(2.6,3.9){\thicklines
    \qbezier(.6,.2)(.75,.25)(.9,.3)}
   \put(2.6,3.9){{\thinlines\color{green}  
  \put(0,0){\line(1,0){.6}}  
  \qbezier(.6,0)(.7,.1)(.9,.3) 
  \qbezier(0,0)(.3,.3)(.9,.9) 
  \qbezier(0,0)(.3,.1)(.6,.2)  
  \qbezier(.6,0)(.7,.3)(.9,.9)  
    } \thicklines\put(0,0){\line(1,0){.4}}
     {\color{blue}
     \qbezier(.4,0)(.5,.1)(.6,.2) } }
     \put(3.9,0){\tetrahedron \thicklines\put(0,0){\line(1,0){.6}}
     {\color{blue}
     \qbezier(.6,0)(.7,.3)(.9,.9) } }
    \put(3.9,1.3){\tetrahedron \thicklines\put(0,0){\line(1,0){.6}}
    {\color{blue}
     \qbezier(.6,0)(.65,.05)(.7,.1)\qbezier(.7,.3)(.8,.6)(.9,.9) }}
    \put(3.9,1.3){\thicklines {\color{red}\put(.7,.1){\line(0,1){.2}} }}
     \put(3.9,2.6){\tetrahedron \thicklines\put(0,0){\line(1,0){.6}} 
     {\color{blue}
     \qbezier(.6,0)(.7,.1)(.8,.2)\qbezier(.8,.6)(.85,.75)(.9,.9) }}
     \put(3.9,2.6){\thicklines {\color{red}\put(.8,.2){\line(0,1){.4}} 
     } }
   \put(3.9,3.9){\tetrahedronB}
    \put(3.9,3.9){
   \thicklines\put(0,0){\line(1,0){.6}}
    {\color{blue}\qbezier(.6,0)(.7,.1)(.9,.3)}}
	 \put(3.9,3.9){\thicklines{\color{red}\put(.9,.3){\line(0,1){.6}}
	 }  }}
\end{picture}}
}}
\label{default}
\end{center}
\end{figure}
%


In general we have the following where we use the notation
\[
	\what{\d_i^-}:I^{k-2}\to I^{k-1}
\]
to indicate the mapping which inserts a 0 between the $i-1$st and $i$th coordinates. The negative $i$th face operator $\d_i^-$ is given by composition with this linear mapping: $\d_i^-(\theta_{(k)})=\theta_{(k)}\circ \what{\d_i^-}$. Analogously, the $i$th face operator $\d_i$ on simplices is given by composition with the linear map whose name is given by putting a hat on the face operation:
\[
	\what{\d_i}=[v_0,\cdots,\what{v_i},\cdots,v_k]:\Delta^{k-1}\to\Delta^k.
\]

\begin{lem}\label{negative face lemma}
For every $1\le i\le k-1$, the following diagram commutes.
\[
\xymatrix{
I^{k-2} \ar[r]^{\what{\d_i^-}}\ar[d]_{\theta_{(k-1)}}  &  I^{k-1}\ar[r]^{\theta_{(k)}\quad} &\Path(\Delta^{k},v_{k},v_0)\\
\Path(\Delta^{k-1},v_{k-1},v_0) \ar[rr]^{\omega_i}  &&  \Path(\Delta^{k-1},v_{k-1},v_0) \ar[u]_{\Path\what{\d_i}}
}
\]
In other words, $\d_i^-(\theta_{(k)})=\theta_{(k)}\circ \what{\d_i^-}=\Path\what{\d_i}\circ\w_i\circ\theta_{(k-1)}$ where, for each $\g\in\Path(\Delta^{k-1},v_{k-1},v_0)$, $\w_i(\g)$ is the reparametrization of $\g$ given as follows (with $j=k-i$).
\[
	\w_{i}(\g)(t)=\begin{cases} \g\left(
	\frac{kt}{k-1}
	\right) & \text{if } t\le\frac{j-1}k\\
	\g\left(
	\frac{j-1}{k-1}
	\right) & \text{if } \frac{j-1}k\le t\le \frac jk\\
	\g\left(
	\frac{kt-1}{k-1}
	\right) & \text{if } t\ge\frac{j}k
	\end{cases}
\]
\end{lem}

This just follows from the definitions. Similarly, we have the following.

\begin{lem}\label{positive face lemma}
For every $1\le i\le k-1$ and $j=k-i$, the following diagram commutes.
\[
\xymatrix{
I^{k-2} \ar[r]^{\what{\d_i^+}}\ar@{=}[d] &  I^{k-1}\ar[r]^{\theta_{(k)}\quad} &\Path(\Delta^{k},v_{k},v_0)\\
I^{i-1}\times I^{j-1} \ar[rr]^{\theta_{(i)}\times\theta_{(j)}\qquad}  &&  \Path(\Delta^{i},v_{i},v_0)\times  \Path(\Delta^{j},v_{j},v_0) \ar[u]_{\mu_{i,j}}
}
\]
In other words, $\d_i^+(\theta_{(k)})=\theta_{(k)}\circ \what{\d_i^+}=\mu_{i,j}\circ(\theta_{(i)}\times\theta_{(j)})$ where $\mu_{i,j}$ is the path composition map
\[
	\mu_{i,j}(\a,\b)(t)=\begin{cases} \what{b_j}\b\left(
	\frac{kt}{j}
	\right) & \text{if } t\le\frac{j}k\\
	\what{f_i}\a\left(
	\frac{k}i(t-\frac jk)
	\right) & \text{if }  t> \frac jk
	\end{cases}
\]
where $\what{f_i}:\Delta^i\to\Delta^k,\what{b_j}:\Delta^j\to\Delta^k$ are the inclusions of the front $i$-face and back $j$-face resp. (so that $f_i\s=\s\circ\what{f_i}$ and $b_j\s=\s\circ \what{b_j}$).
\end{lem}

\begin{rem}
These are the two main properties of the mappings $\th_{(n)}:I^{n-1}\to \Path(\Delta^n,v_n,v_0)$. Chen showed in \cite{Chen77}, p.871, that maps exist with exactly these properties and used them to pull back simplicial classes from $M$ to give cubical classes for the loop space $\Omega M$. Chen gives an explicit formula for this map in his Annals paper \cite{Chen73}, p.240.
\end{rem}


\subsection{Integrating superconnections on simplices}

Suppose that $V\to M$ is a trivial $\ZZ$-graded vector bundle, i.e., $V=\bigoplus V^n$. Let $D$ be a $\ZZ$-graded superconnection on $M$ with coefficients in $V$. Then $D$ can be written as
\[
	D=d-A_0-A_1-\cdots-A_m
\]
where $A_p\in\Omega^p(M,\End^{1-p}(V))$. In particular, $A_0\in\Omega^0(M,\End^1(V))$. For every $v\in M$ let
\[
	\psi_0(v)=d_v=A_0(v)\in \End^1(V).
\]
If $D$ is flat then $d_v^2=0$ and $(V^\ast,d_v)$ is a smooth bundle of cochain complexes over $M$. 

For any smooth simplex $\s:\Delta^k\to M$, $k\ge1$, let
\[
	\psi_k(\s)\in E^{1-k}=\End^{1-k}(V)
\]
be given by
\[
	\psi_k(\s):= \int_{I^{k-1}}\theta_{(k)}^\ast(\Path\s)^\ast\Psi_{k-1}
\]
This is the integral over $I^{k-1}$ of the pull back of the superconnection parallel transport
\[
	\Psi_{k-1}\in\Omega^{k-1}(\Path(M,x_k,x_0),\End^{1-k}(V))
\]
along the composite mapping:
\[
	I^{k-1}\xrarrow{\theta_{(k)}} \Path(\Delta^k,v_k,v_0)\xrarrow{\Path\s}\Path(M,x_k,x_0)
\]
where we use the notation $x_i=\s(v_i)$.
\begin{thm}\label{thm:D is flat iff psi is a twisting cochain}
Let $\psi=\sum_{k\ge0}\psi_k$. Then $D$ is flat if an only if $\psi-u\circ\e$ is an augmented twisting cochain on $C_\ast(M;K)$ with coefficients in $\End(V)$.
\end{thm}

\begin{proof}
Suppose that $D$ is flat. Then, by Theorem 3.5,
\[
	\psi_0(x_0)\psi_k(\s)+(-1)^k \psi_k(\s)\psi_0(x_k)
	=
	\int_{I^{k-1}}\theta_{(k)}^\ast\Path\s^\ast\left(
	A_0\Psi_{k-1}-\Psi_{k-1}A_0
	\right)
\]
\[
	=\int_{I^{k-1}}\theta_{(k)}^\ast \Path\s^\ast\left(
	d\Psi_{k-2}
	\right)
	=\int_{\d I^{k-1}}\theta_{(k)}^\ast \Path\s^\ast
	\Psi_{k-2}
\]
\[
	=\sum_{i=1}^{k-1} (-1)^i\left(
	\int_{{\d_i^-} I^{k-1}}\theta_{(k)}^\ast \Path\s^\ast
	\Psi_{k-2}
	-
	\int_{{\d_i^+} I^{k-1}}\theta_{(k)}^\ast \Path\s^\ast
	\Psi_{k-2}
	\right)
\]
By Lemma \ref{negative face lemma} and invariance under reparametrization (Corollary \ref{invariance under reparametrization}), the $\d_i^-I^{k-1}$ term is
\[
	\int_{{\d_i^-} I^{k-1}}\theta_{(k)}^\ast \Path\s^\ast
	\Psi_{k-2}=
		\int_{I^{k-2}}\left(\what{\d_i^-}\right)^\ast\theta_{(k)}^\ast \Path\s^\ast
	\Psi_{k-2}\]
	\[
	=\int_{I^{k-2}}\theta_{(k-1)}^\ast\Path\what{\d_i}^\ast \Path\s^\ast
	\Psi_{k-2}
	=\int_{I^{k-2}}\theta_{(k-1)}^\ast(\Path\d_i\s)^\ast
	\Psi_{k-2}=\psi_{k-1}(\d_i\s).
\]
By Lemma \ref{positive face lemma}, the $\d_i^+I^{k-1}$ term is
\[
	\int_{{\d_i^+} I^{k-1}}\theta_{(k)}^\ast \Path\s^\ast
	\Psi_{k-2}
	=\int_{I^{k-2}}\left(\what{\d_i^+}\right)^\ast\theta_{(k)}^\ast \Path\s^\ast
	\Psi_{k-2}
\]
\[
	=\int_{I^{i-1}\times I^{j-1}}\left(
	\Path\s\circ\mu_{i.j}\circ(\theta_{(i)}\times\theta_{(j)}
	\right)^\ast \Psi_{k-2}
\]
\[
	=\int_{I^{i-1}}\theta_{(i)}^\ast(\Path f_i\s)^\ast\Psi_{i-1}
	\int_{I^{j-1}}\theta_{(j)}^\ast(\Path b_j\s)^\ast\Psi_{j-1}
	=\psi_i(f_i\s)\psi_j(b_j\s)
\]
where $j=k-i$.
Putting this together, we get:
\[
	\psi_0(x_0)\psi_k(\s)+(-1)^k \psi_k(\s)\psi_0(x_k)
	=\sum_{i=1}^{k-1}(-1)^i \left(\psi_{k-1}(\d_i\s)
	-\psi_i(f_i\s)\psi_j(b_j\s)\right)
\]
or:
\[
	\boxed{\sum_{i=1}^{k-1}(-1)^i\psi_{k-1}(\d_i\s)=\sum_{i=0}^{k}(-1)^i \psi_i(f_i\s)\psi_j(b_j\s)}
\]
which, by Corollary \ref{characterization of augmented twisting cochains}, is equivalent to the statement that $\psi-u\circ\e$ is an augmented twisting cochain.

For the converse, we reverse the argument. Lemmas \ref{negative face lemma}, \ref{positive face lemma} are about maps between cubes and simplices. So they hold regardless of whether $D$ is flat or not. When we arrive at the beginning of the proof, we will know that the integrals are equal and we want to know that the integrands are equal. But if they are not, then we can choose a very small linear simplex to get an inequality of integrals giving a contradiction. So, the converse also holds and $D$ is a flat superconnection if and only if its integral over simplices gives an augmented twisting cochain in the case when $V$ is a trivial bundle.
\end{proof}


\subsection{$A_\infty$ functor}\label{subsec: A-infty functors}

Suppose now that $V=\bigoplus V_n$ is a nontrivial $\ZZ$-graded vector bundle over $M$. Then the structure that we get when we integrate a flat superconnection on $V$ over simplices is not a twisting cochain. It will be an $A_\infty$ functor on the category of simplices of $M$. We give the statement first and the definition afterwards. Note that we switched to lower indices.

The category $Simp(M)$ has objects pairs $(k,\s)$ where $k\ge0$ and $\s$ a smooth mapping $\s:\Delta^k\to M$. A morphism $(j,\t)\to (k,\s)$ is a nondecreasing set map $a:[j]\to [k]$ where $[k]=\{0,1,2,\cdots,k\}$ so that $\t=a^\ast\s=\s\circ \widehat{a^\ast}$ where $\widehat{a^\ast}:\Delta^j\to\Delta^k$ is the linear map which sends $v_i$ to $v_{a(i)}$. (The notation is: $op(\s)=\s\circ\what{op}$.)

Let $D=\nabla-A_0-A_2-A_3-\cdots-A_m$ be a flat superconnection on $V$ and let $\Psi$ be the parallel transport of $D$. Then we claim that the integral of $\Psi$ gives a (contravariant) $A_\infty$ functor on $Simp(M)$. This is just a fancy way of saying the following.

For each object $(n,\s)$ of $Simp(M)$ we get a chain complex
\[
	F_\ast(n,\s)=\left(
	V_{b(\s)},A_0(b(\s))
	\right)
\]
where $b(\s)\in M$ is the barycenter of $\s$. For any morphism $a:(j,\t)\to (k,\s)$ we have a path in $M$ (the image under $\s$ of a line segment in $\Delta^k$) from the barycenter of $\s$ to the barycenter of $\t$. The parallel transport $\Psi_0=\Phi$ gives a chain isomorphism:
\[
	F_1(a):F_\ast(k,\s)\to F_\ast(j,\t)
\]
If $\nabla$ is a flat connection then this defines a contravariant functor from $Simp(M)$ to the category of chain complexes. If not then the integral of the superconnection gives higher homotopies.

Suppose that $\a$ is a $k$ simplex in the nerve of $Simp(M)$, i.e.,
\[
	\a=\left[
	(n_0,\s_0)\to (n_1,\s_1)\to\cdots\to(n_k,\s_k)
	\right]
\]
Then we get a smooth simplex
\[
	\a_\ast:\Delta^k\to M
\]
sending $v_i$ to the barycenter of $\s_i$ and extending linearly. Then we can integrate the superconnection parallel transport over this simplex to get $F_k(\a)$:

\[
	F_k(\a)=\int_{I^{k-1}}\theta_{(k)}^\ast(\Path\a_\ast)^\ast\Psi_{k-1}\in \Hom_{k-1}(F_\ast(n_k,\s_k),F_\ast(n_0,\s_0)).
\]

\begin{cor}\label{D is flat iff its transport is A-infty}
This family of homomorphism $F_k(\a)$ satisfies the formula:
\[
	d_0F_k(\a)+(-1)^kF_k(\a)d_k=\sum_{i=1}^{k-1}(-1)^i
	F_{k-1}(\s_0,\cdots,\what{\s_i},\cdots,\s_k)
	\]
	\[
	-\sum_{i=1}^{k-1}(-1)^iF_{i}(\s_0,\cdots,\s_i)F_{k-i}(\s_i,\cdots,\s_k)
\]
where $d_i=A_0(b(\s_i))$ (i.e., $F$ is an $A_\infty$ functor from $Simp(M)\op$ to the category of chain complexes over $\RR$ as defined below) if and only if $D$ is a flat superconnection on $V$.
\end{cor}

\begin{proof}
Since both statements are local, this corollary is equivalent to its local version which is Theorem \ref{thm:D is flat iff psi is a twisting cochain}. The formula in the corollary is locally the same as the boxed equation in the proof of that theorem.
\end{proof}

\begin{defn}
Suppose that $\cC$ is a differential graded (dg) category over a field $K$, i.e., $\Hom_\cC(A,B)$ is a chain complex over $K$ and composition is a map of chain complexes
\[
	\Hom_\cC(B,C)\otimes\Hom_\cC(A,B)\to\Hom_\cC(A,C).
\]
Let $\cX$ be an ordinary (small) category. Then an \emph{$A_\infty$ functor} $\cX\op\to \cC$ is given by a sequence of operators $F=(F_0, F_1,F_2,\cdots)$ where $F_0$ sends each object $X$ of $\cX$ to an object $F_0X$ of $\cC$ and $F_k$ sends every sequence of $k$ composable morphisms in $\cX\op$:
\[
	X_0\xlarrow{f_1} X_1\xlarrow{f_2}  \cdots \xlarrow{f_k} X_k
\]
to a degree $k-1$ morphism
\[
	F_k(f_1,\cdots,f_k):F_0X_k\to F_0X_0
\]
satisfying the following property.
\[
	dF_k(f_1,\cdots,f_k)=\sum_{i=1}^{k-1}(-1)^i F_{k-1}(f_1,\cdots,f_if_{i+1},\cdots,f_k)
\]
\[
	-\sum_{i=1}^{k-1} (-1)^i F_i(f_1,\cdots,f_i)F_{k-i}(f_{i+1},\cdots,f_k).
\]
\end{defn}

Note that we can replace ``chain complexes'' by ``cochain complexes'' in this definition with the only change being that $F_k(\s)$ has degree $1-k$ for $\s=(f_1,\cdots,f_k)$.



\subsection{Cobar construction}

Suppose that $\cC,\cD$ are dg categories over a field $K$ and $G:\cC\to\cD$ is a \emph{differential graded (dg) functor}, i.e., $G$ induces a degree 0 chain map
\[
	G_\ast:\Hom_\cC(X,Y)\to\Hom_\cD(GX,GY)
\]
for any objects $X,Y$ in $\cC$. Then, clearly, any $A_\infty$ functor $F:\cX\op\to \cC$ gives an $A_\infty$ functor $G\circ F:\cX\op\to\cD$.

Following Getzler and Jones \cite{GetzlerJones89} we want to point out that the \emph{cobar construction} $\cB(\cX\op;K)$ on $\cX\op$ is the target of the universal $A_\infty$ functor on $\cX\op$. In other words, $\cB(\cX\op;K)$ is a dg category with an $A_\infty$ functor $B:\cX\op\to\cB(\cX\op;K)$ having the property that, for any $A_\infty$ functor $F:\cX\op\to\cC$, there is a unique dg functor $F_\ast:\cB(\cX\op;K)\to \cC$ so that $F=F_\ast\circ B$.
\[
\xymatrix{
&  \cB(\cX\op;K) \ar[d]^{\exists !\ F_\ast}\\
\cX\op \ar[r]_{F}\ar[ur]^B  &  \cC 
}
\]
This universal property determines $\cB(\cX\op;K)$ uniquely up to isomorphism. The following standard construction \cite{AdamsCobar} shows that it exists.


\subsubsection{categorical cobar construction}\label{categorical cobar}

Let $\cX$ be any small category and $K$ a field. Then the \emph{cobar construction} $\cB=\cB(\cX\op;K)$ is defined to be the dg category with the same objects as $\cX$ (and $\cX\op$) with morphism sets given by the chain complexes $\cB_\ast(X,Y)$ where $\cB_m(X,Y)$ is the vector space over $K$ generated by all sequences:
\[
	(\s_1|\s_2|\cdots|\s_n):Y=X_0\xlarrow{\s_1}X_1\xlarrow{\s_2}\cdots \xlarrow{\s_n}X_n=X
\]
where each $\s_j$ is a sequence of composable morphisms
\[
	\s_j=(f_1,\cdots,f_k):X_{i-1}\xlarrow{f_1}\bullet\xlarrow{f_2} \cdots\xlarrow{f_{k-1}}\bullet\xlarrow{f_k} X_i
\]
with $|\s_j|:=k-1\ge0$ and $m=|\s_1|+\cdots+|\s_n|$. The boundary map $d:\cB_m(X,Y)\to \cB_{m-1}(X,Y)$ is given by
\[
	d(\s_1|\s_2|\cdots|\s_n)=(d\s_1|\s_2|\cdots|\s_n)+(-1)^{|\s_1|}(\s_1|d\s_2|\cdots|\s_n)+
\]
\[
	\cdots +(-1)^{|\s_1|+\cdots+|\s_{n-1}|}(\s_1|\s_2|\cdots|d\s_n)
\]
where $d\s_j=d(f_1,\cdots,f_k)$ is given by
\[
	d(f_1,\cdots,f_k)=\sum_{i=1}^{k-1}(-1)^i(f_1,\cdots,f_if_{i+1},\cdots,f_k)
	-\sum_{i=1}^{k-1}(-1)^i(f_1,\cdots,f_i|f_{i+1},\cdots,f_k).
\]
One can check that $d^2=0$.

Composition is given by inserting a bar:
\[
	(\s_1|\cdots|\s_n)\circ(\t_1|\cdots|\t_m)=(\s_1|\cdots|\s_n|\t_1|\cdots|\t_m).
\]

Finally, the universal $A_\infty$ functor $B:\cX\op\to\cB(\cX\op;K)$ is given tautologically by $B_0X=X$ for all objects $X$ in $\cX\op$ and $B_k(\s)=(\s)$ for all $\s=(f_1,\cdots,f_k)$. Given any $A_\infty$ functor $F:\cX\op\to\cC$, the uniquely determined dg functor $F_\ast:\cB(\cX;K)\to\cC$ satisfying $F=F_\ast\circ B$ is given by $F_\ast(X)=F_0X$ and
\[
	F_\ast(\s_1|\cdots|\s_n)=F(\s_1)F(\s_2)\cdots F(\s_n).
\]
By the definition of an $A_\infty$ functor we have $dF(\s)=F(d\s)$. This clearly implies that $F_\ast$ is a dg functor.


\subsubsection{topological cobar construction}\label{topological cobar}

Since we are integrating differential forms over smooth simplices we are actually dealing with the \emph{topological cobar construction} given in the smooth category as follows. 

Given a smooth manifold $M$ and a field $K$ let $\cF(M;K)$ be the following dg category. The objects of $\cF(M;K)$ are the points of $M$ and the morphism sets $\cF(x,y)$ are chain complexes where $\cF_m(x,y)$ is the vector space over $K$ spanned by all sequences $(\s_1|\s_2|\cdots|\s_n)$ of smooth simplices in $M$ of dimension $\ge1$ so that the first vertex of $\s_1$ is $y$, the last vertex of $\s_n$ is $x$ and the first vertex of each $\s_j$ for $j\ge2$ is equal to the last vertex of $\s_{j-1}$. The boundary map is given by
\[
	d(\s_1|\s_2|\cdots|\s_n)=(d\s_1|\s_2|\cdots|\s_n)+(-1)^{|\s_1|}(\s_1|d\s_2|\cdots|\s_n)+
\]
\[
	\cdots +(-1)^{|\s_1|+\cdots+|\s_{n-1}|}(\s_1|\s_2|\cdots|d\s_n)
\]
where $|\s|=\dim \s-1$ and $m=|\s_1|+\cdots+|\s_n|$ and $d\s$ for $\s:\Delta^k\to M$ is given by
\[
	d\s=\sum_{i=1}^{k-1}(-1)^i\d_i\s-\sum_{i=1}^{k-1}(-1)^i (f_i\s|b_{n-i}\s).
\]
Composition is given by inserting a bar:
\[
	(\s_1|\cdots|\s_n)\circ(\t_1|\cdots|\t_m)=(\s_1|\cdots|\s_n|\t_1|\cdots|\t_m).
\]

\begin{cor}\label{dg functor on cobar}
Suppose that $V=\bigoplus V_n$ is a graded vector bundle over $M$ and $D$ is a flat superconnection on $M$ with coefficients in $\End(V)$. Then a dg functor $\psi_\ast$ from the cobar construction $\cF(M;K)$ to the dg category of finitely generated chain complexes over $K$ is given as follows. On objects, i.e. points, $x\in M$ the functor is given as before by $\psi_\ast(x)=V_x$ with differential $\d=A_0(x)$. On morphisms it is given by the composition
\[
	\psi_\ast(\s_1|\cdots|\s_n)=\psi(\s_1) \psi(\s_2)\cdots \psi(\s_n)
\]
where $\psi(\s)$ for $\s:\Delta^k\to M$ is given by integrating the parallel transport of $D$:
\[
	\psi(\s)=\int_{I^{k-1}}\theta_{(k)}^\ast(\Path\s)^\ast\Psi_{k-1}.
\]
\end{cor}

\begin{proof}
The boxed equation in the proof of Theorem \ref {thm:D is flat iff psi is a twisting cochain} says 
\[
\psi(d\s)=\d \circ \psi(\s)-(-1)^{|\s|}\psi(\s)\circ\d
\]
where $\d=A_0$. This implies $\psi_\ast$ is a chain map.
\end{proof}

%
%

\section{Relation to the work of K-T Chen}\label{sec 5: the work of Chen}

In a series of influential papers in the 70's K-T Chen constructed iterated integrals of ``formal connections'' on a smooth manifold and used them to obtain many results related to the fundamental group and the homology of the loop space of the manifold. The introduction to Chen's collected works \cite{Chen} gives a very nice overview. In this very short discussion our aim is to point out the extent to which $\ZZ$-graded superconnections and their parallel transport is contained in Chen's work. \cite{Chen75},\cite{Chen73}, \cite{Chen77}. Chen worked simultaneously with real and complex manifolds and he used the notation $k=\RR$ or $\CC$. For simplicity we restrict to the real case.


\subsection{Formal connections}

In \cite{Chen75} Chen defined a \emph{connection} on a smooth manifold $M$ with coefficients in a finite dimensional vector bundle $V$ to be an operator
\[
	D:\Omega(M,V)\to \Omega(M,V)
\]
of positive degree with the property that, for any $\w\in\Omega(M)$, $v\in\Gam V$, 
\[
	D(\w \otimes v)=d\w\otimes v+J\w\otimes D(1\otimes v)
\]
where $J\w=(-1)^{|\w|}\w$. If $D$ can be interpreted as an \emph{odd} operator then this is the supercommutator rule
\[
	[D,\w]=d\w.
\]
In most of his papers Chen does this, i.e. makes $D$ odd, formally in the following way.

\begin{defn}
Chen defined a \emph{formal power series connection}, or \emph{formal connection} for short, to be a differential form on $M$ with coefficients in the graded formal power series ring $\RR\<\<X_1,\cdots,X_m\>\>$ in noncommuting variables $X_i$ having degree $\ge0$ of the form
\[
	\w=\sum w_i X_i+\sum w_{ij}X_iX_j+\sum w_{ijk}X_iX_jX_k+\cdots
\]
where 
\[
\deg w_{j_1\cdots j_n}=1+\deg X_{j_1}X_{j_2}\cdots X_{j_n}=1+\sum \deg X_{j_i}
\]
We would say that $D=d-\w$ is a $\ZZ$-graded superconnection of total degree 1 if the degrees of the formal variables are inverted (their signs are changed).
\end{defn}

Chen makes the comment in \cite{Chen73} that any graded algebra over $\RR$ is a quotient of the formal power series algebra $\RR\<\<X_1,\cdots,X_m\>\>$ by a homogenous ideal.

Chen defines the \emph{curvature} of his formal connection $\w$ to be
\[
	\k:=d\w-J\w\wedge \w
\]
He notes that $\k=0$ if and only if $D^2=0$ where $D=d-\w$. 


\subsection{Parallel transport}

Chen constructs the parallel transport of his formal power series connection, obtaining a differential form on the path space. He calls this $T$. In our terminology $T=\Psi$. Since $A_0$ does not enter into the construction of the parallel transport, our construction is a special case of Chen's construction. Chen also shows that, if $\w$ is flat, the parallel transport is closed: $dT=0$. In our terminology this is the statement $d\Psi_k=0$. This follows from the fact that Chen is taking $A_0=0$ in this case.

The Bismut-Lott definition of flat superconnection involves the variable differential $A_0$. This is what Chen called a ``differential twisting cochain'' in \cite{Chen73}. In \cite{Chen75} Chen considers a connection $\w$ on $M$ with coefficients in a DGA $(A_\ast,\d)$ so that $\d \w+\k=0$. In our notation $\d=A_0$ which Chen takes to be fixed. Chen proves the following.

\begin{thm} Let $C_\ast(\Omega M)$ be the normalized differentiable cubical chain complex of the loop space $\Omega M$ of $M$ at a point $x_0$. Let $D=d-\w$ be a (super)connection on $M$ with coefficients in a DGA $(A_\ast,\d)$ so that $\d\w+\k=0$. Then,
integration of the parallel transport $T$ over cubical simplices gives a chain map
\[
	\Theta:C_\ast(\Omega M)\to (A_\ast,\d).
\]
\end{thm}

Chen calls this chain map $\Theta$ a ``generalized holonomy map''. 

Consider the special case where $A_\ast=\End(V)$ with differential $\d$ given by commutation with an element $\delta$ of degree $-1$. Then $D'=d-\w-\delta$ becomes a flat superconnection. So, this special case of the above theorem follows from Theorem~\ref{thm:flat superconnections give higher homotopies} (1). (However, Chen's proof was a lot shorter!)

Chen also shows that this generalized holonomy map induces an isomorphism in homology: $H_\ast(\Omega M)\cong H_\ast(A_\ast,\d)$ if $A_\ast$ is chosen appropriately. However, the homology of the loop space does not concern us.


\subsection{Cubes to simplices}

In \cite{Chen73}, Chen constructs mappings
\[
	\th_{(n)}:I^{n-1}\to \Path(\Delta^n,v_n,v_0)
\]
satisfying our Lemmas \ref{negative face lemma} and \ref{positive face lemma}.
In \cite{Chen77}, he shows that any map having these properties can be used to pull back simplicial classes from $M$ to give cubical classes for the loop space $\Omega M$. Chen assumes that $\th_{(n)}$ is smooth. This is accomplished by slowing down the piecewise linear mappings at the corners in the usual way. Since reparametrizations of a path do not change the holonomy, the details of the smoothing are irrelevant.

Suppose that $M$ is path connected and $x_0\in M$. Then Chen considers the subcomplex $C=\Delta(M)_{x_0}$ of the singular chain complex of $M$ spanned by smooth simplices $\Delta^k\to M$ so that all vertices of $\Delta^k$ go to $x_0$. The cobar construction $F(C)$ on $C$ is just the tensor algebra of $C_{>0}$ desuspended once. If $D$ is a (super)connection on $M$ with coefficients in a DGA $(A_\ast,\d)$ satisfying the twisting cochain condition $\d \w+\k=0$ then, as we explained in subsection \ref{topological cobar}, Chen obtains a chain map
\[
	F(C)\to A_\ast
\]
Chen uses this to reinterpret and prove some variations of Adams' theorem \cite{AdamsCobar}. He shows in particular that, when $M$ is simply connected,
\[
	H_\ast(F(C))\cong H_\ast(\Omega M).
\]

%
\section{Higher Reidemeister torsion}\label{sec 6: higher FR torsion}
%

Very briefly, the construction of higher Reidemeister torsion using Morse theory and flat superconnections is a follows. (See \cite{IComplexTorsion} for an outline of this process and see \cite{IBookOne} for complete details.) Suppose we have a smooth bundle $E\to B$ and a hermitian coefficient system $\cF$ on $E$ so that the homology of each fiber with coefficients in $\cF$ is trivial. For example, a lens space bundle might have this property. A circle bundle might also have this property.

In that case, we construct a canonical fiberwise generalized Morse function using the framed function theorem \cite{IFF}. This gives a family of acyclic chain complexes $C(b),b\in B$. Using the 2-index theorem (originally due to Hatcher and explained in detail in \cite{IBookOne} in this, the chain complex case) we can reduce the family of chain complexes to a family of acyclic based chain complexes which are nonzero in only two degrees. In other words, we have a family of invertible matrices. We smoothly interpolate to get a smooth family of invertible matrices $g_b,b\in B$. In the language of superconnections, $g_b=A_0(b)$. This extends uniquely to a flat superconnection $D=d-A_0-A_1-A_2$ with higher terms being zero since there are no degree 3 endomorphisms for a complex in two consecutive degrees. The higher Reidemeister torsion in degree $2k$ is given by the differential form on $B$:
\[
	\frac1{(2k+1)!2i^k}\int_{t=0}^t Tr\left((h_b^{-t}dh_b^t)^{2k+1}\right)+\text{correction terms}
\]
where $h_b=g_bg_b^\ast$. Thus the main term depends only on $A_0=g_b$. The higher terms are in the correction terms which are polynomials in $A_0,A_1,A_2$. We use the fact that the higher torsion invariant is a polylogarithm and polylogarithms are linearly independent from polynomial functions. This allows us to ignore the correct terms. A recursive formula for these higher correction terms is given in \cite{IBookOne}.

\end{document}